\documentclass[11pt]{article}
\usepackage{amsmath,fullpage,amssymb,amsthm,hyperref,enumerate}
\usepackage{tikz,color,cases}
\usetikzlibrary{patterns}
\usepackage{marvosym,ifthen}
\usepackage{graphicx}
\usepackage{stmaryrd}
 
\newtheorem{theorem}{Theorem}[section]

\newtheorem{lemma}[theorem]{Lemma}

\theoremstyle{definition}

\theoremstyle{remark}

\renewcommand{\P}{\mathcal P}
\newcommand{\G}{\mathcal G}

\newcommand{\Z}{\mathbb Z}
\DeclareMathOperator\des{des}
\DeclareMathOperator\maj{maj}

\newcommand{\bfc}{\mathbf c}
\newcommand{\bfd}{\mathbf d}
\newcommand{\bfe}{\mathbf e}
\newcommand{\bff}{\mathbf f}
\newcommand{\qbin}[2]{\begin{bmatrix}{#1}\\ {#2}\end{bmatrix}_q}

\newcommand{\al}{\alpha}
\newcommand{\be}{\beta}
\newcommand{\ga}{\gamma}
\newcommand{\de}{\delta}

\DeclareMathOperator\cro{\chi}
\newcommand\start{circle(.07)}
\newcommand\startud{circle(.07)}
\newcommand\E{-- ++(1,0) \start}
\newcommand\N{-- ++(0,1) \start}
\newcommand\dn{-- ++(1,-1) \startud}
\newcommand\up{-- ++(1,1) \startud}
\newcommand{\si}{\sigma}
\newcommand{\ta}{\tau}

\newcommand{\ci}{\circ} 
\newcommand{\bu}{\bullet}

\newcommand{\su}[1]{\lVert#1\rVert}
\newcommand{\uw}{\uparrow}
\newcommand{\dw}{\downarrow}
\newcommand{\swap}{\varsigma}
\newcommand{\alt}{<_{\mathrm{alt}}}

\newenvironment{list1}{
  \begin{list}{}{%
      \setlength{\itemsep}{2mm}
      \setlength{\parsep}{1mm} \setlength{\parskip}{1mm}
      \setlength{\topsep}{0mm} \setlength{\partopsep}{0in}
      \setlength{\leftmargin}{0in}}}{\end{list}}

\newcommand\pathsP[5]{\P_{#1\to #2,#3\to #4}^{\ge #5}}

\newcommand\markx[2]{
\draw (#1,.2)--(#1,-.2) node[below] {$#1$};
\draw (.2,#2)--(-.2,#2) node[left] {$#2$};
\draw[dotted] (#1,.3)--(#1,#2)--(.3,#2);
}

\newcommand\crossing[2]{\draw[thick] (#1,#2) circle (.2);}
\newcommand\crossingc[3]{\draw[thick,#3] (#1,#2) circle (.2);}
\newcommand\crossingdot[2]{\draw[thick,dotted] (#1,#2) circle (.2);}

\newcommand\diamant[2]{\draw[teal] (#1-.25,#2)--(#1,#2+.25)--(#1+.25,#2)--(#1,#2-.25)--(#1-.25,#2);}

\newcommand\movepeakvalley[2]{
\diamant{#1}{#2} 
\diamant{#1}{-#2} 
\ifthenelse{#2=0}{}{\draw[teal,dotted,thick,->,shorten <=4pt,shorten >=4pt] (#1,#2)--(#1,-#2);}
}

\newcommand\tra[2]{\genfrac\{\}{0pt}{1}{#1}{#2}}
\newcommand\tr[2]{\genfrac{}{}{0pt}{1}{#1}{#2}}
\newcommand\ptra[4]{\genfrac{\{}{\vert}{0pt}{1}{#1}{#2}\genfrac{.}{\}}{0pt}{1}{#3}{#4}}
\newcommand\ptr[4]{\genfrac{.}{\vert}{0pt}{1}{#1}{#2}\genfrac{}{}{0pt}{1}{#3}{#4}}
\newcommand\cd{\tr{\bfc}{\bfd}}
\newcommand\ef{\tr{\bfe}{\bff}}
\newcommand\cdef{\ptr{\bfc}{\bfd}{\bfe}{\bff}}

\newcommand\Tr{T}

\newcommand{\lessup}{{\rotatebox[origin=c]{45}{$<$}}}
\newcommand{\lessdown}{{\rotatebox[origin=c]{-45}{$<$}}}
\newcommand{\uwcrossing}[1]{\draw (#1+.5,.8) circle (.25);}

\newcommand{\uwcrossingdot}[1]{\draw[dotted] (#1+.5,.8) circle (.25);}
\newcommand{\dwcrossingdot}[1]{\draw[dotted] (#1,0) circle (.25);}
\newcommand{\uwcrossingc}[2]{\draw[#2] (#1+.5,.8) circle (.25);}
\newcommand{\dwcrossingc}[2]{\draw[#2] (#1,0) circle (.25);}
\newcommand{\uwcrossingfill}[1]{\draw[fill=yellow] (#1+.5,.8) circle (.25);}

\newcommand{\uwcrossingfillc}[2]{\draw[#2,fill=yellow] (#1+.5,.8) circle (.25);}
\newcommand{\dwcrossingfillc}[2]{\draw[#2,fill=yellow] (#1,0) circle (.25);}

\newcommand\drawarray[2]{
\foreach \ai [count=\i] in #1
	{\draw (\i-.5,.8) node {$\ai$};}
\foreach \ai [count=\i] in #2
	{\draw (\i-1,0) node {$\ai$};}
}
\newcommand{\ineq}[2]{
\foreach \i in {2,...,#1} {\draw (\i-1,.8) node {$<$};}
\draw (#1,0.8) node {$\le$};
\draw (.5,0) node {$\le$};
\foreach \i in {2,...,#2} {\draw (\i-.5,0) node {$<$};}
}
\newcommand{\ineqx}[2]{
\foreach \i in {1,...,#1} {\draw (\i,.8) node {$<$};}
\draw (.5,0) node {$\le$};
\foreach \i in {3,...,#2} {\draw (\i-1.5,0) node {$<$};}
\draw (#2-.5,0) node {$\le$};
}
\newcommand{\ineqeq}[2]{
\foreach \i in {2,...,#1} {\draw (\i-1,.8) node {$<$};}
\draw (#1,0.8) node {$\le$};
\draw (.5,0) node {$=$};
\foreach \i in {2,...,#2} {\draw (\i-.5,0) node {$<$};}
}
\newcommand{\ineqxeq}[2]{
\foreach \i in {1,...,#1} {\draw (\i,.8) node {$<$};}
\draw (.5,0) node {$=$};
\foreach \i in {3,...,#2} {\draw (\i-1.5,0) node {$<$};}
\draw (#2-.5,0) node {$\le$};
}

\author{Sergi Elizalde}

\title{Counting lattice paths by crossings and major index II:\\ tracking descents via two-rowed arrays}

\date{}

\begin{document}

\maketitle

\begin{abstract}
We present refined enumeration formulas for lattice paths in $\Z^2$ with two kinds of steps, by keeping track of the number of descents (i.e., turns in a given direction), the major index (i.e., the sum of the positions of the descents), and the number of crossings.
One formula considers crossings between a path and a fixed line; the other considers crossings between two paths. 
Building on the first paper of the series, which used lattice path bijections to give the enumeration with respect to major index and crossings, we obtain a refinement that keeps track of the number of descents. The proof is based on new bijections which rely on certain two-rowed arrays that were introduced by Krattenthaler.
\end{abstract}

\section{Introduction}\label{sec:intro}

\subsection{Background}

Lattice paths in the plane with two kinds of steps have played an important role in combinatorics and mathematical statistics for decades~\cite{Moh,Krat}. The statistic giving the number of times that a path crosses a fixed line has been studied at least since the sixties~\cite{Eng,Sen,Feller,Feller-book,KW,Spivey}, often in connection to random walks. For tuples of paths, the enumeration in the special case of non-crossing tuples, in its closely related non-intersecting variant, is given by the celebrated 
Gessel--Viennot determinant \cite{GV}, also discovered by Lindstr\"om~\cite{Lin} in the context of matroid theory, and 
has applications to symmetric functions, plane partitions, tilings, and statistical physics \cite{Fisher}. 

On the other hand, a very different statistic, the sum of the positions of the turns in a given direction, has been studied in~\cite{KM, Krat-nonint,SaSa}. This statistic is called the {\em major index} because it arises naturally
when interpreting the paths as binary words, and it was introduced by MacMahon~\cite{Mac}.

In the first paper of this series~\cite{part1}, we enumerated paths with respect to the number of crossings of a line and the major index, as well as pairs of paths with respect to the number of times they cross each other and the sum of their major indices.
 The goal of the present paper is to refine the results from~\cite{part1} by another important statistic, which is related to the major index and arguably more natural: the number of turns in a given direction, or equivalently, the number of descents of the associated binary word. 
 
 The number of turns arises when studying the distribution of runs in random walks~\cite{Moh}, the coefficients of Hilbert polynomials of determinantal and Pfaffian rings~\cite{Kul}, and summations for Schur functions~\cite{Krat-nonint}. 
A thorough investigation of this parameter on lattice paths was provided by Krattenthaler~\cite{Krat-turns}. In particular,
a refinement by this statistic of the classical determinantal formula of Gessel--Viennot~\cite{GV} counting tuples of non-intersecting paths was given in \cite[Thm.\ 1]{Krat-turns0} and \cite[Thm.\ 3.6.1]{Krat-turns}. In related work, Krattenthaler and Mohanty~\cite{KM} enumerated lattice paths constrained to a strip with respect to the number of descents and the major index. 

The tools that were used in~\cite{part1} to deal with crossings and the major index consisted of bijections with a neat description in terms of lattice paths. While these bijections were suited to study the major index, unfortunately they do not behave well with respect to the number of descents, which is why the results obtained in~\cite{part1} do not include this statistic. Instead, in this paper we will construct different bijections that are not described in terms of paths, but rather in terms of two-rowed arrays.
Such arrays, which are more general than paths, have been used by Krattenthaler and Mohanty to study descents and major index on lattice paths in a strip~\cite{KM}, and by Krattenthaler to enumerate tuples of non-intersecting paths with respect to the number of turns~\cite{Krat-turns0,Krat-turns} and to the major index~\cite{Krat-nonint}.
However, to our knowledge, they have never been used while also keeping track of the number of crossings. 
While two-rowed arrays allow us to track simultaneously track multiple statistics, including the number of descents, 
the tradeoff is that they make the proofs more involved and less intuitive than those in~\cite{part1}.

Paralleling the results in~\cite{part1}, this paper solves two problems: the enumeration of single paths with respect to the number of times that they cross a fixed line, and the enumeration of pairs of paths with respect to the number of times that they cross each other, refined in both cases by the number of descents and the major index. This paper is self-contained and does not rely on any material from~\cite{part1}.

Our work is partially motivated by the simplicity of the resulting formulas in both cases.
For single paths with given endpoints, crossing a line at least a certain number of times and having a fixed number of descents, we will show that the polynomial enumerating them with respect to the major index is given by a product of two $q$-binomial coefficients and a power of $q$. For pairs of paths crossing each other, the formulas we obtain involve a product of two generating functions whose coefficients have again the same form.

The second source of motivation is that our results for paths crossing a line have applications to the refined enumeration of integer partitions according to the number of sign changes of their successive ranks (or off-diagonal ranks). These applications, which generalize results of Seo and Yee~\cite{SeoYee}, will be explored in~\cite{CES} in connection to the study of partitions with constrained ranks.

\subsection{Preliminaries}\label{sec:basic}

For points $A,B\in\Z^2$, we denote by $\P_{A\to B}$ the set of lattice paths with steps $N=(0,1)$ (north) and $E=(1,0)$ (east)
that start at $A$ and end at $B$.
Sometimes it will be convenient to consider paths with steps $U=(1,1)$ (up) and $D=(1,-1)$ (down) instead. 
For nonnegative integers $a,b$, we denote by $\G_{a,b}$ set of paths with $a$ steps $U$ and $b$ steps $D$ starting at the origin.

In both cases, encoding paths as binary words, with $0$s recording $N$ (resp.\ $U$) steps, and $1$s recording $E$ (resp.\ $D$) steps, we define a {\em descent} (also called a {\em valley}) of the path to be a vertex preceded by an $E$ and followed by an $N$ (resp.\ preceded by a $D$ and followed by a $U$). The number of descents of a path $P$ is denoted by $\des(P)$.
The major index of  $P$, denoted by $\maj(P)$, is defined to be the sum of the positions of the descents, where the position is determined by numbering the vertices along the path, starting at 0. See Figure~\ref{fig:Gstrl} for an example.
We also define a {\em peak} of the path to be a vertex preceded by an $N$ and followed by an $E$ (resp.\ preceded by a $U$ and followed by a $D$).

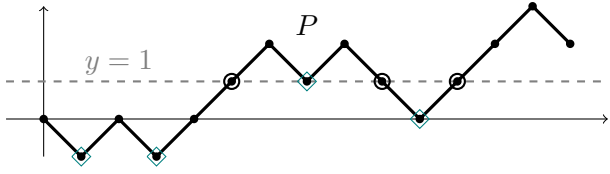
\begin{figure}[htb]
\centering
\begin{tikzpicture}[scale=0.5]
\draw[->] (-1,0)--(15,0);
\draw[->] (0,-1)--(0,3);
\draw[dashed,thick,gray] (-1,1)--(15,1);
\draw (2,1.5) node[gray] {$y=1$};
\draw[very thick,fill] (0,0) \startud \dn\up\dn\up\up\up\dn\up\dn\dn\up\up\up\dn;
\draw (7,2.5) node {$P$};
\crossing{5}{1}
\crossing{9}{1}
\crossing{11}{1}
\diamant{1}{-1}
\diamant{3}{-1}
\diamant{7}{1}
\diamant{10}{0}
\end{tikzpicture}
\caption{A path $P\in\G_{8,6}^{\ge 3,1}$ with $\maj(P)=1+3+7+10=21$. The four valleys are marked with teal diamonds, and the
three crossings of the line $y=1$ are circled in black. The middle crossing is a downward crossing, whereas the other two are upward crossings.}
\label{fig:Gstrl}
\end{figure}

The enumeration of binary words by the number of descents and the major index is implicit in work of MacMahon~\cite{Mac}. An explicit proof was given by F\"urlinger and Hofbauer~\cite{FH}. To state this result in its lattice path version, recall that the {\em $q$-binomial coefficients} are defined as
$$\qbin{m}{n}=\frac{(1-q^m)(1-q^{m-1})\cdots(1-q^{m-n+1})}{(1-q^n)(1-q^{n-1})\cdots(1-q)}
$$
if $0\le n\le m$, and as $0$ otherwise.

\begin{lemma}[\cite{Mac,FH}]\label{lem:qbin2}
For $a,b\ge0$,
$$\sum_{P\in\G_{a,b}}t^{\des(P)}q^{\maj(P)}=\sum_{n\ge0} t^n q^{n^2}\qbin{a}{n}\qbin{b}{n}.$$
Equivalently, for $x,y,u,v\in\Z$,
$$\sum_{P\in\P_{(x,y)\to (u,v)}} t^{\des(P)} q^{\maj(P)}=\sum_{n\ge0} t^n q^{n^2}\qbin{u-x}{n}\qbin{v-y}{n}.$$
\end{lemma}

A self-contained proof of this lemma will be included in Section~\ref{sec:2ra}.
The rest of the paper is structured as follows. In Section~\ref{sec:main} we state our results, both for single paths crossing a line and for pairs of paths crossing each other. In Section~\ref{sec:proofs-single} we prove them in the case of single paths crossing a line, by introducing two-rowed arrays to encode paths, generalizing the notion of crossings to such arrays, and then describing certain bijections on them.
In Section~\ref{sec:proofs-pairs} we prove our results for pairs of paths crossing each other, by generalizing crossings to pairs of two-rowed arrays, and then defining bijections on such pairs.

\section{Main results}\label{sec:main}

\subsection{Paths crossing a line}

First we consider the enumeration of paths with $U$ and $D$ steps according to the number of times that they cross a fixed horizontal line. For integers $\ell,r$ with $r\ge0$,
let $\G_{a,b}^{\ge r,\ell}$ denote the set of paths in $\G_{a,b}$ that cross the line $y=\ell$ at least $r$ times.
A vertex of the path on the line $y=\ell$ is a {\em crossing} if it is either preceded and followed by a $D$ ---in which case it is called a {\em downward crossing}---, or preceded and followed by a $U$ ---called an {\em upward crossing}. See Figure~\ref{fig:Gstrl} for an example.

We will provide expressions for the polynomials
$$G_{a,b}^{\ge r,\ell}(t,q)=\sum_{P\in\G_{a,b}^{\ge r,\ell}} t^{\des(P)} q^{\maj(P)}$$
for arbitrary integers $a,b,r,\ell$ with $a,b,r\ge0$. 
Note that the polynomials for paths crossing the line $y=\ell$ {\em exactly} $r$ times can be obtained from the above simply as $G_{a,b}^{\ge r,\ell}(t,q)-G_{a,b}^{\ge r+1,\ell}(t,q)$.

An expression for $G_{a,b}^{\ge r,\ell}(1,q)$ was given in \cite[Theorems 2.1 and 2.2]{part1}. The following result refines these theorems by incorporating the statistic $\des$.

\begin{theorem}\label{thm:line:des}
Let $a,b,m\ge0$, and let $\ell\in\mathbb{Z}$.
\begin{enumerate}[I.]
\item If $0<\ell<a-b$, then
\begin{equation}\label{eq:0<l<a-b:des}
G_{a,b}^{\ge 2m+1,\ell}(t,q)=G_{a,b}^{\ge 2m,\ell}(t,q)=\sum_{n\ge0}t^n q^{n^2+m(m+\ell+1)}\qbin{a}{n-m}\qbin{b}{n+m}.
\end{equation}
\item If $0>\ell>a-b$, then
\begin{equation}\label{eq:0>l>a-b:des}
G_{a,b}^{\ge 2m+1,\ell}(t,q)=G_{a,b}^{\ge 2m,\ell}(t,q)=\sum_{n\ge0}t^n q^{n^2+m(m-\ell-1)}\qbin{a}{n+m}\qbin{b}{n-m}.
\end{equation} 
\item If $0>\ell<a-b$, then
\begin{equation}\label{eq:0>l<a-b:des}
G_{a,b}^{\ge 2m+2,\ell}(t,q)=G_{a,b}^{\ge 2m+1,\ell}(t,q)=\sum_{n\ge0}t^n q^{n^2+(m+1)(m-\ell)}\qbin{a-\ell-1}{n-m-1}\qbin{b+\ell+1}{n+m+1}.
\end{equation}
\item If $0<\ell>a-b$, then
\begin{equation}\label{eq:0<l>a-b:des}
G_{a,b}^{\ge 2m+2,\ell}(t,q)=G_{a,b}^{\ge 2m+1,\ell}(t,q)=\sum_{n\ge0}t^n q^{n^2+m(m+\ell+1)}\qbin{a-\ell-1}{n+m}\qbin{b+\ell+1}{n-m}.
\end{equation}
\item If $0=\ell<a-b$, then
\begin{align}\label{eq:0=l<a-b:des-even}
G_{a,b}^{\ge 2m,\ell}(t,q)&=\sum_{n\ge0}t^n q^{n^2+m(m+1)}\qbin{a}{n-m}\qbin{b}{n+m},\\
\label{eq:0=l<a-b:des-odd}
G_{a,b}^{\ge 2m+1,\ell}(t,q)&=\sum_{n\ge0}t^n q^{n^2+m(m+1)}\qbin{a-1}{n-m-1}\qbin{b+1}{n+m+1}.
\end{align}
\item If $0=\ell>a-b$, then
\begin{align}\label{eq:0=l>a-b:des-even}
G_{a,b}^{\ge 2m,\ell}(t,q)&=\sum_{n\ge0}t^n q^{n^2+m(m-1)}\qbin{a}{n+m}\qbin{b}{n-m},\\
\label{eq:0=l>a-b:des-odd}
G_{a,b}^{\ge 2m+1,\ell}(t,q)&=\sum_{n\ge0}t^n q^{n^2+m(m+1)}\qbin{a-1}{n+m}\qbin{b+1}{n-m}.
\end{align}
\item If $0<\ell=a-b$, then
\begin{align}\label{eq:0<l=a-b:des-even}
G_{a,b}^{\ge 2m,\ell}(t,q)&=\sum_{n\ge0}t^n q^{n^2+m(m+\ell+1)}\qbin{a}{n-m}\qbin{b}{n+m}, \\
\label{eq:0<l=a-b:des-odd}
G_{a,b}^{\ge 2m+1,\ell}(t,q)&=\sum_{n\ge0}t^n q^{n^2+m(m+\ell+1)}\qbin{a+1}{n-m}\qbin{b-1}{n+m}.
\end{align}
\item If $0>\ell=a-b$, then
\begin{align}\label{eq:0>l=a-b:des-even}
G_{a,b}^{\ge 2m,\ell}(t,q)&=\sum_{n\ge0}t^n q^{n^2+m(m-\ell-1)}\qbin{a}{n+m}\qbin{b}{n-m}, \\
\label{eq:0>l=a-b:des-odd}
G_{a,b}^{\ge 2m+1,\ell}(t,q)&=\sum_{n\ge0}t^n q^{n^2+(m+1)(m-\ell)}\qbin{a+1}{n+m+1}\qbin{b-1}{n-m-1}.
\end{align}
\item If $0=\ell=a-b$, then
\begin{align}\label{eq:0=l=a-b:des-even}
G_{a,b}^{\ge 2m,\ell}(t,q)&=\sum_{n\ge0}t^n q^{n^2+m(m+1)}\frac{1-q^{a-2m}}{1-q^a}\qbin{a}{n+m}\qbin{a}{n-m}, \\ 
\label{eq:0=l=a-b:des-odd}
G_{a,b}^{\ge 2m+1,\ell}(t,q)&=\sum_{n\ge0}t^n q^{n^2+m(m+1)}\frac{1-q^{a+2(m+1)}}{1-q^a}\qbin{a}{n+m+1}\qbin{a}{n-m-1}.
\end{align}
\end{enumerate}
\end{theorem}

\subsection{Pairs of paths crossing each other}\label{sec:pairs-crossing}

Next we consider the enumeration pairs of paths with respect to the number of crossings between them. For this problem it is convenient to consider paths with $N$ and $E$ steps. Let $P$ and $Q$ be two such paths, and suppose that $V_1,V_2,\dots,V_s$ (where $s\ge1$) is a maximal sequence of consecutive common vertices such that
\begin{itemize}
\item neither $V_1$ nor $V_s$ are endpoints of $P$ or $Q$;
\item for each of $P$ and $Q$, its step arriving at $V_1$ is of the same type ($N$ or $E$) as its step leaving~$V_s$.
\end{itemize}
In this case, vertex $V_s$ is called a {\em crossing} of $P$ and $Q$.
This definition differs slightly from the one used in~\cite{part1}, where the term crossing refers to the first vertex $V_1$ of the sequence. Of course, the number of crossings of $P$ and $Q$ does not depend on this convention, but defining the crossing to be $V_s$ will be more convenient in the proofs in Section~\ref{sec:proofs-pairs}.
Figure~\ref{fig:crossing} shows some examples of crossings.

\begin{figure}[htb]
\centering
\begin{tikzpicture}[scale=0.6]
    \draw[red,very thick,fill](0,1) \start\E\E\N\E;
    \draw[blue,dashed,very thick,fill](1,0) \start\N\E\N\N;
\crossing{2}{2}
    \draw[red,very thick,fill](7,0) \start\N\E\N\E\N;
    \draw[blue,dashed,very thick,fill](6,1)\start\E\E\E\N\E;
\crossing{8}{1}    
    \draw[red,very thick,fill](15,0) \start\N\E\E;
    \draw[blue,dashed,very thick,fill](14,1)\start\E\E\N;
\end{tikzpicture}
\caption{Two examples of crossings, circled in black, and a pair of paths that do not cross (right).}
\label{fig:crossing}
\end{figure}
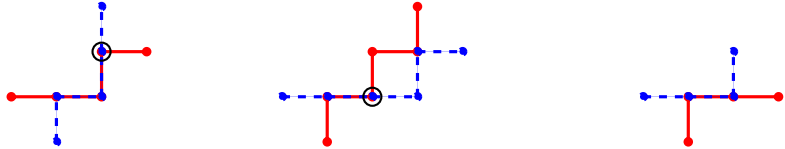

Let $\cro(P,Q)$ denote the number of crossings of  $P$ and $Q$; see Figure~\ref{fig:pair} for an example. For $A_1,A_2,B_1,B_2\in\Z^2$, $r\ge0$, and $\{\ci,\bu\}=\{1,2\}$, let
$$\pathsP{A_1}{B_\ci}{A_2}{B_\bu}{r}=\{(P,Q):P\in\P_{A_1\to B_\ci},Q\in\P_{A_2\to B_\bu},\cro(P,Q)\ge r\}.$$

\begin{figure}[htb]
\centering
\begin{tikzpicture}[scale=0.6]
\diamant{3}{2}
\diamant{6}{4}
\diamant{4}{5}
\diamant{7}{6}
\diamant{8}{7}
    \draw[red,very thick,fill](0,2) \start\E\E\E\N\N\E\E\E\N\N\N\E\E\E\E;
    \draw[blue,dashed,very thick,fill](2,0) \start\N\N\E\N\N\N\E\N\E\E\E\N\E\N;
\crossing{3}{4};\crossing{6}{6};\crossing{8}{7};	
	\draw (0,2) node[left] {$A_1$};
	\draw (2,0) node[left] {$A_2$};     
	\draw (8,8) node[right] {$B_1$};
	\draw (10,7) node[right] {$B_2$};     
	    \draw[red] (1,2.5) node {$P$} ;
    \draw[blue] (2.5,1) node {$Q$} ;
\end{tikzpicture}
\caption{A pair of paths with $\cro(P,Q)=3$, $\des(P)+\des(Q)=6$, and $\maj(P)+\maj(Q)=45$.}
\label{fig:pair}
\end{figure}

To enumerate such pairs of paths with respect to the sum of their numbers of descents (the {\em total descent number}) and the sum of their major indices (the {\em total major index}), we define the polynomials
$$H^{\ge r}_{A_1\to B_\ci,A_2\to B_\bu}(t,q)= \sum_{(P,Q)\in\pathsP{A_1}{B_\ci}{A_2}{B_\bu}{r}} t^{\des(P)+\des(Q)}q^{\maj(P)+\maj(Q)}.$$
Note that the polynomials for pairs of paths that cross each other {\em exactly} $r$ times are given by the difference $H^{\ge r}_{A_1\to B_\ci,A_2\to B_\bu}(t,q)-H^{\ge r+1}_{A_1\to B_\ci,A_2\to B_\bu}(t,q)$.

To state our formulas, let us first define the following polynomial in $t$ and $q$ that depends on the points $A_1=(x_1,y_1)$, $A_2=(x_2,y_2)$, $B_1=(u_1,v_1)$, $B_2=(u_2,v_2)$, and a parameter $k\in\Z$: 
\begin{multline}\label{eq:fr}
f_{k,A_1,A_2,B_2,B_1}(t,q)\\
=q^{k(k+x_2-x_1)}\left(\sum_{n\ge0} t^n q^{n(n+k)} \qbin{u_2-x_1}{n} \qbin{v_2-y_1}{n+k}\right)
\left(\sum_{n\ge0} t^n q^{n(n-k)} \qbin{u_1-x_2}{n} \qbin{v_1-y_2}{n-k}\right).
\end{multline}
The theorem below refines~\cite[Theorem 2.4]{part1}. 

\begin{theorem}\label{thm:pairs:des}
Let $A_1=(x_1,y_1)$, $A_2=(x_2,y_2)$, $B_1=(u_1,v_1)$ and $B_2=(u_2,v_2)$ be points in $\Z^2$ such that $A_1\prec A_2$ and $B_1\prec B_2$. Suppose additionally that 
\begin{equation}
\label{condition}
x_1+y_1=x_2+y_2.
\end{equation} 
Then, for all $m\ge0$,
\begin{align}\label{eq:switched:des}
H^{\ge 2m+1}_{A_1\to B_2,A_2\to B_1}(t,q)&=H^{\ge 2m}_{A_1\to B_2,A_2\to B_1}(t,q)=f_{2m,A_1,A_2,B_2,B_1}(t,q),\\
\label{eq:same:des}
H^{\ge 2m+2}_{A_1\to B_1,A_2\to B_2}(t,q)&=H^{\ge 2m+1}_{A_1\to B_1,A_2\to B_2}(t,q)=f_{2m+1,A_1,A_2,B_2,B_1}(t,q).
\end{align}
Let now $A=(x,y)$ and $B=(u,v)$ be points in $\Z^2$. Then, for all $r\ge0$,
\begin{align}
\label{eq:A1=A2:des}
H^{\ge r}_{A\to B_1,A\to B_2}(t,q)&=f_{r,A,A,B_2,B_1}(t,q),\\
\label{eq:B1=B2:des}
H^{\ge r}_{A_1\to B,A_2\to B}(t,q)&=f_{r,A_1,A_2,B,B}(t,q),\\
\label{eq:A1=A2,B1=B2:des}
H^{\ge r}_{A\to B,A\to B}(t,q)&=\begin{cases}f_{0,A,A,B,B}(t,q) & \text{if }r=0,\\
2\sum_{j\ge1}(-1)^{j-1}f_{r+j,A,A,B,B}(t,q) & \text{if }r\ge1.
\end{cases}
\end{align}
\end{theorem}

\section{Proofs for paths crossing a line}\label{sec:proofs-single}

In this section we prove Theorem~\ref{thm:line:des}. Before diving into the details, we remark that it would be possible to give an alternative proof by induction on the length (number of steps) of the path, by first separating each of the nine cases of the theorem into two subcases, according to whether the last step of the path is a $U$ or a $D$. For example, if $0<\ell<a-b$, the refinement to be proved by induction would state that the generating function for paths in $\G_{a,b}^{\ge r,\ell}$ that end with a $D$, where $r=2m$ or $r=2m+1$, equals
$$G_{a,b-1}^{\ge r,\ell}(t,q)=\sum_{n\ge0}t^n q^{n^2+m(m+1+\ell)}\qbin{a}{n-m}\qbin{b-1}{n+m},$$
and so the generating function for those that end with a $U$ equals
$$G_{a,b}^{\ge r,\ell}(t,q)-G_{a,b-1}^{\ge r,\ell}(t,q)=\sum_{n\ge0}t^n q^{n^2+m(m+\ell)+b-n}\qbin{a}{n-m}\qbin{b-1}{n+m-1}.$$
Then, to prove each one of these formulas, we would remove the last step of the path, and deduce them from the formulas for shorter paths that hold by the induction hypothesis. This often requires additional subcases; for example, for the above paths ending in $U$, the cases $\ell+1<a-b$ and $\ell+1=a-b$ would be considered separately.

Instead of such a tedious induction proof, we have chosen to present a proof that relies on certain two-rowed arrays that have been used by Krattenthaler and Mohanty. One advantage of our proof is that it is bijective.  Additionally, the methodology of two-rowed arrays that we introduce here will later allow us to prove Theorem~\ref{thm:pairs:des} for pairs of paths, where a potential proof by induction is much less clear.

\subsection{Two-rowed arrays}\label{sec:2ra}

Let $x,y,u,v,k\in\Z$ and $n,j\ge0$ throughout the section. We use the notation
\begin{align*}
(x,u]_j&=\{(c_1,\dots,c_j):x<c_1<c_2<\dots<c_j\le u\},\\
[y,v)_j&=\{(a_1,\dots,a_j):y\le d_1<d_2<\dots<d_j<v\},\\
(x,v)_j&=\{(c_1,\dots,c_j):x<c_1<c_2<\dots<c_j<v\},\\
[y,u]_j&=\{(a_1,\dots,a_j):y\le d_1<d_2<\dots<d_j\le u\}.
\end{align*}
We consider pairs of such sequences arranged in a particular way, which we call {\em two-rowed arrays}, following~\cite{Krat-nonint,Krat-turns0,Krat-turns,KM}. We denote by $\tra{(x,u]_{n+k}}{[y,v)_{n-k}}$, or $\tra{(x,u]}{[y,v)}_{n\pm k}$ for short, the set of arrays of the form
\arraycolsep=2pt
$$\begin{array}{cccccccccccccc}
&x&<&c_1&<&c_2&<&&\dots&&<&c_{n+k}&\le&u\\
y&\le &d_1&<&d_2&<&\dots&<&d_{n-k}&<&v&&&
\end{array},$$
with the convention that this set is empty unless $|k|\le n$. 
The two rows are interlaced from the left, starting with the leftmost element in the bottom row.
Elements in this set are denoted by $\cd$, where
$\bfc=(c_1,\dots,c_{n+k})\in(x,u]_{n+k}$ and $\bfd=(d_1,\dots,d_{n-k})\in[y,v)_{n-k}$.

Similarly, we denote by $\tra{(x,v)}{[y,u]}_{n\pm k}$ the set of arrays of the form
\arraycolsep=2pt
$$\begin{array}{cccccccccccccc}
&x&<&c_1&<&c_2&<&&\dots&&<&c_{n+k}&<&v\\
y&\le &d_1&<&d_2&<&\dots&<&d_{n-k}&\le&u&&&
\end{array},$$

The reason two-rowed arrays are useful for our problem is that elements of $\tra{(x,u]}{[y,v)}_{n\pm 0}$, which we denote simply by $\tra{(x,u]}{[y,v)}_{n}$, encode lattice paths in $\P_{(x,y)\to (u,v)}$. This is because such paths are uniquely determined by the coordinates of their valleys. There exists a path in $\P_{(x,y)\to (u,v)}$ whose valleys are at coordinates $(c_1,d_1), (c_2,d_2), \dots, (c_n,d_n)$ if and only if
$$x< c_1<c_2<\dots<c_n\le u \quad\text{and}\quad y\le d_1<d_2<\dots<d_n< v,$$
that is, $\bfc=(c_1,\dots,c_n)\in(x,u]_n$ and $\bfd=(d_1,\dots,d_n)\in[y,v)_n$. 
Thus, this encoding is a bijection 
\begin{equation}\label{bij:Pto2ra}\{P\in\P_{(x,y)\to (u,v)}:\des(P)=n\}\to\tra{(x,u]}{[y,v)}_{n}.\end{equation}
It has the property that, if $P$ is encoded by $\cd$, then
\begin{equation}\label{eq:maj2ra}\maj(P)=\sum_{i=1}^n (c_i+d_i-x-y)=\su\bfc+\su\bfd-n(x+y),\end{equation}
where $\su\bfc$ denotes the sum of the entries of $\bfc$.
Next we enumerate two-rowed arrays with respect to this statistic.

\begin{lemma}\label{lem:sum}
\begin{enumerate}[(i)]
\item We have
\begin{align*}
\sum_{\bfc\in(x,u]_j}q^{\su\bfc}&=q^{\binom{j+1}{2}+jx}\qbin{u-x}{j}, &
\sum_{\bfd\in[y,v)_j}q^{\su\bfd}&=q^{\binom{j+1}{2}+j(y-1)}\qbin{v-y}{j},\\
\sum_{\bfc\in(x,v)_j}q^{\su\bfc}&=q^{\binom{j+1}{2}+jx}\qbin{v-x-1}{j}, &
\sum_{\bfd\in[y,u]_j}q^{\su\bfd}&=q^{\binom{j+1}{2}+j(y-1)}\qbin{u-y+1}{j}.
\end{align*}
\item We have
\begin{align}
\sum_{\cd\in\tra{(x,u]}{[y,v)}_{n\pm k}} q^{\su\bfc+\su\bfd-n(x+y)}
&=q^{n^2+k(k+x-y+1)}\qbin{u-x}{n+k}\qbin{v-y}{n-k},
\label{eq:sum1}\\
\sum_{\cd\in\tra{(x,v)}{[y,u]}_{n\pm k}} q^{\su\bfc+\su\bfd-n(x+y)}
&=q^{n^2+k(k+x-y+1)}\qbin{v-x-1}{n+k}\qbin{u-y+1}{n-k}
\label{eq:sum2}
\end{align}
\end{enumerate}
\end{lemma}

\begin{proof} We prove the first identity in part (i), since the other three are analogous.
Writing $c'_i=c_i-i-x$ for $1\le i\le j$, the left-hand side is equal to
$$\sum_{x< c_1<c_2<\dots<c_j\le u}q^{c_1+\dots+c_j}=q^{\binom{j+1}{2}+jx}\sum_{0\le c'_1\le c'_2\le \dots\le c'_j\le u-x-j}q^{c'_1+\dots+c'_j}.$$
This sum counts partitions with at most $j$ parts with largest part at most $u-x-j$, which is a well-known interpretation of the $q$-binomial coefficients (see e.g.~\cite{Andrews}).

Part (ii) follows easily from part (i) using the simplification
$$\binom{n+k+1}{2}+\binom{n-k+1}{2}+(n+k)x+(n-k)(y-1)-n(x+y)=n^2+k(k+x-y+1)$$
in the exponent of $q$.
\end{proof}

To see how Lemma~\ref{lem:sum} will be applied, let us first use it to give a proof of Lemma~\ref{lem:qbin2}. 

\begin{proof}[Proof of Lemma~\ref{lem:qbin2}]
The two statements are clearly equivalent, so we prove the second one. Using the encoding~\eqref{bij:Pto2ra}, together with Equations~\eqref{eq:maj2ra} and~\eqref{eq:sum1} for $k=0$, we get
\[
\sum_{P\in\P_{(x,y)\to (u,v)}} t^{\des(P)} q^{\maj(P)}=\sum_{n\ge0}t^n\sum_{\cd\in\tra{(x,u]}{[y,v)}_{n}} q^{\su\bfc+\su\bfd-n(x+y)}=\sum_{n\ge0}t^nq^{n^2}\qbin{u-x}{n}\qbin{v-y}{n}.\qedhere
\]
\end{proof}

\subsection{Crossings in single two-rowed arrays}\label{sec:crossings}

To encode paths in $\G_{a,b}$ as two-rowed arrays, we first turn the $U$ and $D$ steps into $N$ and $E$ steps, respectively. Additionally, to study crossings of the line $y=\ell$ in the original path, we move the starting point to $(\ell,0)$, so that these crossings become crossings of the diagonal $y=x$ for the resulting path. Denoting by $\P_{A\to B}^{\ge r}$ the set of paths in $\P_{A\to B}$ that cross the diagonal at least $r$ times,
this transformation is a bijection 
\begin{equation}\label{bij:GtoP} \G_{a,b}^{\ge r,\ell}\to\P_{(\ell,0)\to(b+\ell,a)}^{\ge r}. \end{equation} 
See Figure~\ref{fig:single-path-NE} for an example.
In analogy to the definitions for paths in $\G_{a,b}$ crossing a line $y=\ell$, we define upward (resp. downward) crossings of paths in $\P_{A\to B}$ to be vertices in the diagonal $y=x$ that are preceded and followed by an $N$ (resp. by an $E$).

\begin{figure}[htb]
\centering
\begin{tikzpicture}[scale=0.6]
\draw (-.3,0)--(8,0);
\draw (0,-.3)--(0,8.5);
\markx{1}{0}
\markx{2}{0}
\markx{3}{1}
\markx{4}{4}
\markx{6}{5}
\markx{7}{8}
\draw[dashed,thick,gray] (-.5,-.5)--(8,8);
\draw (1.5,2.5) node[gray] {$y=x$};
\draw[very thick,fill](1,0) \start\E\N\E\N\N\N\E\N\E\E\N\N\N\E;
\crossingc{3}{3}{olive};\crossingc{5}{5}{violet};\crossingc{6}{6}{orange};
\draw (1,0) node[below left] {$A$};
\draw (7,8) node[above right] {$B$};
\draw (3,2.5) node[right] {$P$} ;
\diamant{2}{0}
\draw[teal,above right=-1mm] (2,0) node[scale=.8] {$(2,0)$};
\diamant{3}{1}
\draw[teal, right=1mm] (3,1) node[scale=.8] {$(3,1)$};
\diamant{4}{4}
\draw[teal,below right] (4,4) node[scale=.8] {$(4,4)$};
\diamant{6}{5}
\draw[teal,below right] (6,5) node[scale=.8] {$(6,5)$};
\draw[<->] (8.75,4)--(9.5,4);
\begin{scope}[shift={(11,3.5)},scale=1.17]
\drawarray{{1,2,3,4,6,7}}{{0,0,1,4,5,8}}
\ineq{5}{5}
\uwcrossingc{2}{olive}
\uwcrossingc{4}{orange}
\dwcrossingc{4}{violet}
\draw[right] (5.7,0.4) node {$\in\tra{(1,7]}{[0,8)}_{4}$};
\end{scope}
\end{tikzpicture}
\caption{The path in $\P_{(1,0)\to(7,8)}^{\ge 3}$ obtained by applying the transformation~\eqref{bij:GtoP} to the path in Figure~\ref{fig:Gstrl}, and the corresponding two-rowed array given by the encoding~\eqref{bij:Pto2ra}, where the crossings have been circled.}
\label{fig:single-path-NE}
\end{figure}
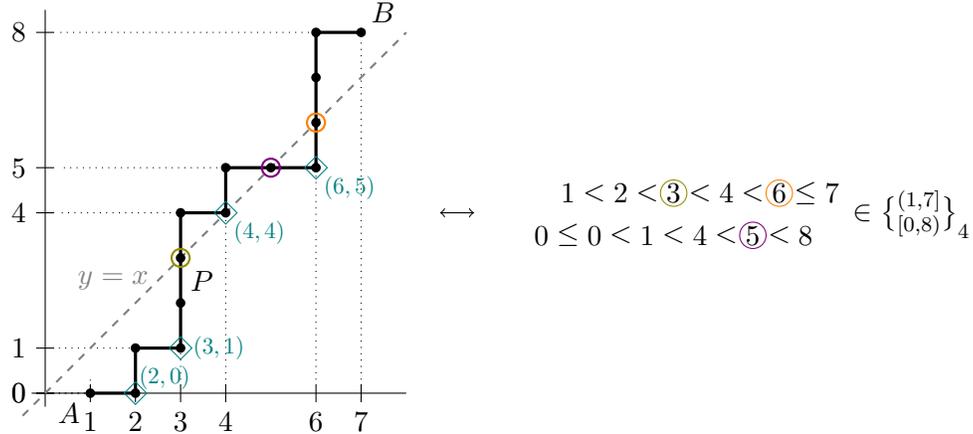

Next we show how these crossings of the diagonal can be read from the encoding~\eqref{bij:Pto2ra} of the path as a two-rowed array. Indeed, suppose that 
$P\in\P_{(x,y)\to (u,v)}$ is encoded by $\cd\in\tra{(x,u]}{[y,v)}_{n}$, and let $c_0:=x$, $d_0:=y$, $c_{n+1}:=u$, $d_{n+1}:=v$ by convention.
An upward crossing of $P$ occurs when, for some $0\le i\le n$, the vertex $(c_i,d_i)$ ---which is a valley or the first vertex of the path--- lies below the diagonal and the vertex $(c_i,d_{i+1})$ ---which is a peak or the last vertex of the path--- lies above the diagonal. This happens precisely when $d_i<c_i<d_{i+1}$ for some $0\le i\le n$.
Similarly, a downward crossing occurs when, for some $1\le i\le n+1$, the vertex $(c_{i-1},d_i)$ ---which is a peak or the starting point of the path--- lies above the diagonal and the vertex $(c_i,d_i)$ ---which is a valley or the last vertex of the path--- lies below the diagonal. This happens precisely when $c_{i-1}<d_i<c_i$ for some $1\le i\le n+1$.

This description allows us to extend the notion of crossings to two-rowed arrays $\cd\in\tra{(x,u]}{[y,v)}_{n\pm k}$ with $k\in\Z$,
whose rows may have different lengths. Using the convention $c_0:=x$, $d_0:=y$, $c_{n+k+1}:=u$, $d_{n-k+1}:=v$, say that $\cd$ has an {\em upward crossing} at $c_i$ if $0\le i\le n-|k|$ and
\begin{equation}\label{eq:upward} d_i<c_i<d_{i+1},\end{equation}
and that it has a {\em downward crossing} at $d_i$ if $1\le i\le n-|k|+1$ and
\begin{equation}\label{eq:downward}  c_{i-1}<d_i<c_i. \end{equation} 
For two-rowed arrays of the form $\cd\in\tra{(x,v)}{[y,u]}_{n\pm k}$, the definition of upward and downward crossings is the same, using now the convention  $c_0:=x$, $d_0:=y$, $c_{n+k+1}:=v$, $d_{n-k+1}:=u$. 
Figure~\ref{fig:truncated} shows two examples, where the crossings have been circled. As usual, the term {\em crossings} refers to both upward and downward crossings.

\begin{figure}[htb]
\centering
\begin{tikzpicture}[scale=0.6]
\begin{scope}[shift={(-2,6.5)},scale=1.17]
\drawarray{{0,1,4,5,7,7}}{{2,2,3,5}}
\ineq{5}{3}
\dwcrossingc{2}{olive}
\uwcrossingc{2}{violet}
\draw[right] (5.7,0.4) node {$\in\tra{(0,7]}{[2,5)}_{3\pm 1}$};
\end{scope}

\draw (-.3,0)--(5.5,0);
\draw (0,-.3)--(0,5.5);
\draw[dashed,thick,gray] (-.5,-.5)--(5.5,5.5);
\draw[very thick,fill](0,2) \start\E\N\E\E\E\N\N\E;
\crossingc{3}{3}{olive};\crossingc{4}{4}{violet};
\diamant{1}{2}
\draw[teal,below right] (1,2) node[scale=.8] {$(1,2)$};
\diamant{4}{3}
\draw[teal,below right] (4,3) node[scale=.8] {$(4,3)$};
\draw[left] (0,2) node[scale=.8] {$(0,2)$};
\draw[right] (5,5) node[scale=.8] {$(5,5)$};

\begin{scope}[shift={(13,0)}]
\begin{scope}[shift={(-2,6.5)},scale=1.17]
\drawarray{{0,3,4,6}}{{0,1,2,5,7,8}}
\ineqx{3}{5}
\dwcrossingc{1}{olive}
\uwcrossingc{2}{violet}
\dwcrossingc{3}{orange}

\draw[right] (5.2,0.4) node {$\in\tra{(0,6)}{[0,8]}_{3\mp 1}$};
\end{scope}

\draw (-.3,0)--(6.5,0);
\draw (0,-.3)--(0,5.5);
\draw[dashed,thick,gray] (-.5,-.5)--(5.5,5.5);
\draw[very thick,fill](0,0) \start\N\E\E\E\N\E\N\N\N\E\E;
\crossingc{1}{1}{olive};\crossingc{4}{4}{violet};\crossingc{5}{5}{orange};
\diamant{3}{1}
\draw[teal,below right] (3,1) node[scale=.8] {$(3,1)$};
\diamant{4}{2}
\draw[teal,below right] (4,2) node[scale=.8] {$(4,2)$};
\draw[left] (0,0) node[scale=.8] {$(0,0)$};
\draw[right] (6,5) node[scale=.8] {$(6,5)$};
\end{scope}
\end{tikzpicture}
\caption{Two two-rowed arrays $\cd$ with their crossings circled, and the corresponding paths $\Tr(\cd)$. Note that for the array on the right, $c_3=6$ is not a crossing because it violates the condition $i\le n-|k|$.}
\label{fig:truncated}
\end{figure}

In both of the above cases, let $\Tr(\cd)\in\P_{(x,y)\to(c_{n-|k|+1},d_{n-|k|+1})}$ be the path whose valleys are at coordinates $(c_i,d_i)$ for $1\le i\le n-|k|$ (with the caveat that, in the special case when $\cd\in\tra{(x,v)}{[y,u]}_{n}$ and $d_n=u$, the vertex $(c_n,d_n)$ is not actually a valley of this path). Then the upward and downward crossings of the two-rowed array $\cd$ can be identified with the upward and downward crossings of $\Tr(\cd)$; see the examples in Figure~\ref{fig:truncated}. Note that $\Tr(\cd)$ is essentially the path corresponding to the two-rowed array obtained by truncating the longer row of $\cd$ so that both rows have equal length. To be precise, this path depends not only on $\cd$ but also on the endpoints $x,y,u,v$, 

Throughout the paper, the {\em $r$th crossing} of a two-rowed array refers to the $r$th crossing {\em from the left}, in the order in which the entries are placed, namely $y,x,d_1,c_1,d_2,c_2,\dots$.
We note that this convention is different from the one used in~\cite{part1}, where path crossings were numbered from the right. The unusual convention in~\cite{part1} was needed because the path bijections in that paper, in order to track the major index, changed the portion of the paths to the {\em left} of a crossing. On the other hand, the notation in this paper becomes slightly simpler by defining bijections for two-rowed arrays (in Sections~\ref{sec:al-be} and~\ref{sec:ga-de}) that change the portion of the arrays to the {\em right} of a crossing instead.

For nonnegative $r$, the superscript ${\ge}r$ on a set of two-rowed arrays denotes the subset of those that have at least $r$ crossings. When $r\ge1$, a symbol $\uw$ (resp.\ $\dw$) next to this superscript denotes the subset where the $r$th crossing is an upward (resp.\ downward) crossing.
For example, 
$\tra{(x,u]}{[y,v)}_{n\pm k}^{\ge r\uw}$ consists of two-rowed arrays in $\tra{(x,u]}{[y,v)}_{n\pm k}^{\ge r}$ where the $r$th crossing is an upward crossing. In the case $r=0$, we simply define 
\begin{equation}\label{eq:convention_r0}
\tra{(x,u]}{[y,v)}_{n\pm k}^{\ge 0\uw}=\tra{(x,u]}{[y,v)}_{n\pm k}^{\ge 0\dw}=\tra{(x,u]}{[y,v)}_{n\pm k}^{\ge 0}=\tra{(x,u]}{[y,v)}_{n\pm k}
\end{equation}
by convention.

The encoding~\eqref{bij:Pto2ra} restricts to a bijection 
\begin{equation}\label{bij:Pto2ra:cro}
\{P\in\P_{(x,y)\to (u,v)}^{\ge r}:\des(P)=n\}\to\tra{(x,u]}{[y,v)}_{n}^{\ge r}.
\end{equation}
Composing this with the bijection~\eqref{bij:GtoP} and using Equation~\eqref{eq:maj2ra}, it follows that
\begin{equation}\label{eq:Gto2ra}
G_{a,b}^{\ge r,\ell}(t,q)=\sum_{n\ge0}\,t^n\sum_{\cd\in\tra{(x,u]}{[y,v)}_{n}^{\ge r}} q^{\su\bfc+\su\bfd-n(x+y)},
\end{equation}
where $(x,y)=(\ell,0)$ and $(u,v)=(b+\ell,a)$.

To prove Theorem~\ref{thm:line:des}, we will construct bijections between $\tra{(x,u]}{[y,v)}_{n}^{\ge r}$ and sets of the form 
$\tra{(x,u]}{[y,v)}_{n\pm k}$ or $\tra{(x,v)}{[y,u]}_{n\pm k}$ for some $k\in\Z$, which will depend on the relations between $x$ and $y$ and between $u$ and $v$, and then apply Lemma~\ref{lem:sum}.

\begin{lemma}\label{lem:first_crossing}
Let $r\ge1$, and let $\cd$ be a two-rowed array in either $\tra{(x,u]}{[y,v)}_{n\pm k}^{\ge r}$ or $\tra{(x,v)}{[y,u]}_{n\pm k}^{\ge r}$.\\
If $x>y$ or $x=y=d_1$, then the $r$th crossing of $\cd$ is an upward crossing if $r$ is odd, and a downward crossing if $r$ is even.\\
If $x<y$ or $x=y<d_1$, then the $r$th crossing of $\cd$ is a downward crossing if $r$ is odd, and an upward crossing if $r$ is even.
\end{lemma}

\begin{proof}
As noted above, upward and downward crossings of $\cd$ are the same as those of the path $\Tr(\cd)\in\P_{(x,y)\to(c_{n-|k|+1},d_{n-|k|+1})}$.
If $x>y$ (resp.\ $x<y$), this path starts below (resp.\ above) the diagonal, which forces the first crossing to be upward (resp.\ downward), with successive crossings alternating between upward and downward. If $x=y$, then $\Tr(\cd)$ starts with an $E$ if $y=d_1$, and with an $N$ if $y<d_1$, from which the same conclusions follow.
\end{proof}

The next lemma shows that the relationships between $x$ and $y$ and between $u$ and $v$ often force the number of crossings of a two-rowed array to have a given parity. We use the notation $n\mp s$ to mean $n\pm(-s)$.

\begin{lemma}\label{lem:even=odd}
Let $s,m\ge0$. \begin{enumerate}[(a)]
\item If $x>y$ and $u<v$, then
\begin{align}\label{eq:x>y,u<v1} & \tra{(x,u]}{[y,v)}_{n\pm s}^{\ge 2m+1}=\tra{(x,u]}{[y,v)}_{n\pm s}^{\ge 2m+1\uw}=\tra{(x,u]}{[y,v)}_{n\pm s}^{\ge 2m\dw}=\tra{(x,u]}{[y,v)}_{n\pm s}^{\ge 2m},&\\
\label{eq:x>y,u<v2}
& \tra{(x,v)}{[y,u]}_{n\mp s}^{\ge 2m+2}=\tra{(x,v)}{[y,u]}_{n\mp s}^{\ge 2m+2\dw}=\tra{(x,v)}{[y,u]}_{n\mp s}^{\ge 2m+1\uw}=\tra{(x,v)}{[y,u]}_{n\mp s}^{\ge 2m+1}.&
\end{align}
\item If $x>y$ and $u>v$, then
\begin{align}\label{eq:x>y,u>v1} & \tra{(x,u]}{[y,v)}_{n\mp s}^{\ge 2m+2}=\tra{(x,u]}{[y,v)}_{n\mp s}^{\ge 2m+2\dw}=\tra{(x,u]}{[y,v)}_{n\mp s}^{\ge 2m+1\uw}=\tra{(x,u]}{[y,v)}_{n\mp s}^{\ge 2m+1},&\\
\label{eq:x>y,u>v2}
&\tra{(x,v)}{[y,u]}_{n\pm s}^{\ge 2m+1}=\tra{(x,v)}{[y,u]}_{n\pm s}^{\ge 2m+1\uw}=\tra{(x,v)}{[y,u]}_{n\pm s}^{\ge 2m\dw}=\tra{(x,v)}{[y,u]}_{n\pm s}^{\ge 2m}.&
\end{align}
\item If $x<y$ and $u<v$, then
\begin{align}\label{eq:x<y,u<v1}  &\tra{(x,u]}{[y,v)}_{n\pm s}^{\ge 2m+2}=\tra{(x,u]}{[y,v)}_{n\pm s}^{\ge 2m+2\uw}=\tra{(x,u]}{[y,v)}_{n\pm s}^{\ge 2m+1\dw}=\tra{(x,u]}{[y,v)}_{n\pm s}^{\ge 2m+1},&\\
\label{eq:x<y,u<v2}
&\tra{(x,v)}{[y,u]}_{n\mp s}^{\ge 2m+1}=\tra{(x,v)}{[y,u]}_{n\mp s}^{\ge 2m+1\dw}=\tra{(x,v)}{[y,u]}_{n\mp s}^{\ge 2m\uw}=\tra{(x,v)}{[y,u]}_{n\mp s}^{\ge 2m}.&
\end{align}
\item If $x<y$ and $u>v$, then
\begin{align}\label{eq:x<y,u>v1} &\tra{(x,u]}{[y,v)}_{n\mp s}^{\ge 2m+1}=\tra{(x,u]}{[y,v)}_{n\mp s}^{\ge 2m+1\dw}=\tra{(x,u]}{[y,v)}_{n\mp s}^{\ge 2m\uw}=\tra{(x,u]}{[y,v)}_{n\mp s}^{\ge 2m},&\\
\label{eq:x<y,u>v2}
&\tra{(x,v)}{[y,u]}_{n\pm s}^{\ge 2m+2}=\tra{(x,v)}{[y,u]}_{n\pm s}^{\ge 2m+2\uw}=\tra{(x,v)}{[y,u]}_{n\pm s}^{\ge 2m+1\dw}=\tra{(x,v)}{[y,u]}_{n\pm s}^{\ge 2m+1}.&
\end{align}
\item If $x>y$ and $u=v$, then  \eqref{eq:x>y,u<v1}, \eqref{eq:x>y,u<v2}, \eqref{eq:x>y,u>v1} hold for $s\ge1$, and \eqref{eq:x>y,u>v2} holds for $s\ge0$.
\item If $x<y$ and $u=v$, then  \eqref{eq:x<y,u<v1}, \eqref{eq:x<y,u<v2}, \eqref{eq:x<y,u>v1} hold for $s\ge1$, and \eqref{eq:x<y,u>v2} holds for $s\ge0$.
\item Statements (a),(b),(e) also hold if we replace $x>y$ with $x=y$ and restrict to 
two-rowed arrays $\cd$ with $y=d_1$.
\end{enumerate}
\end{lemma}

\begin{proof}
In each equation, the outer equalities follow from Lemma~\ref{lem:first_crossing} (using the convention~\eqref{eq:convention_r0} as needed), and the left-hand side is trivially contained in the right-hand side. To prove the reverse containment, we will show that the parity of the number of crossings of the relevant two-rowed arrays is determined by the relation between $x$ and $y$ and between $u$ and $v$ in each case.

Recall that if $\cd$ is a two-rowed array in either $\tra{(x,u]}{[y,v)}_{n\pm k}$ or $\tra{(x,v)}{[y,u]}_{n\pm k}$, for some $k\in\Z$, then $\Tr(\cd)$ is a path from $(x,y)$ to $(c_{n-|k|+1},d_{n-|k|+1})$ which has the same upward and downward crossings as~$\cd$. The parity of the number of crossings is determined by what side of the diagonal the endpoints of the path are on.
If $x>y$, $\Tr(\cd)$ starts below the diagonal; if $x<y$, it starts above the diagonal; and if $x=y=d_1$, it starts with an $E$ leaving the diagonal, so it behaves as in the $x>y$ case.

Suppose first that $\cd\in\tra{(x,u]}{[y,v)}_{n\pm s}$, where $s\ge0$ and $u<v$. Then the last vertex of $\Tr(\cd)$ is $(c_{n-s+1},v)$, which lies above the diagonal, since $c_{n-s+1}\le u<v$. Thus, if $x>y$, then $\Tr(\cd)$ starts below the diagonal and ends above the diagonal, so it must have an odd number of crossings, proving Equation~\eqref{eq:x>y,u<v1}. If $x=y=d_1$, the same conclusion holds. On the other hand, if $x<y$, then $\Tr(\cd)$ starts and ends above the diagonal, so it must have an even number of crossings, proving Equation~\eqref{eq:x<y,u<v1}. 
Modifying the hypotheses so that $s\ge1$ and $u\ge v$, the last vertex of $\Tr(\cd)$ still lies above the diagonal, since $c_{n-s+1}<c_{n-s+2}\le u\le v$, so Equations~\eqref{eq:x>y,u<v1} and~\eqref{eq:x<y,u<v1} also hold in this case.

If $\cd\in\tra{(x,v)}{[y,u]}_{n\mp s}$, where $s\ge0$ and $u<v$, then the last vertex of $\Tr(\cd)$ is $(v,d_{n-s+1})$, which lies below the diagonal, since $d_{n-s+1}\le u<v$. Thus, $\Tr(\cd)$ must have an even number of crossings if $x>y$ or $x=y=d_1$, proving Equation~\eqref{eq:x>y,u<v2}, and an odd number of crossings if $x<y$, proving Equation~\eqref{eq:x<y,u<v2}. 
These two equations still hold with the modified hypotheses $s\ge1$ and $u\le v$, since $d_{n-s+1}<d_{n-s+2}\le u\le v$ in this case, so the last vertex of $\Tr(\cd)$ still lies below the diagonal.

If $\cd\in\tra{(x,u]}{[y,v)}_{n\mp s}$, where $s\ge0$ and and $u>v$, then the last vertex of $\Tr(\cd)$ is $(u,d_{n-s+1})$, which lies below the diagonal, since $d_{n-s+1}\le v< u$. This vertex also lies below the diagonal when $s\ge1$ and and $u\ge v$, since $d_{n-s+1}<d_{n-s+2}\le v\le u$.
This proves Equations~\eqref{eq:x>y,u>v1} and~\eqref{eq:x<y,u>v1}.

If $\cd\in\tra{(x,v)}{[y,u]}_{n\pm s}$, where $s\ge0$ and $u>v$, then the last vertex of $\Tr(\cd)$ is $(c_{n-s+1},u)$, which lies above the diagonal, since $c_{n-s+1}\le v< u$. 
This vertex also lies above the diagonal when $s\ge1$ and $u\ge v$, since $c_{n-s+1}<c_{n-s+2}\le v\le u$.
Finally, when $s=0$ and $u=v$, the path $\Tr(\cd)$ ends on the diagonal (at $(v,u)$), but its last step is an $E$ step, since $c_{n-s}<v$.
This proves Equations~\eqref{eq:x>y,u>v2} and~\eqref{eq:x<y,u>v2} for all $s\ge0$ and $u\ge v$.
\end{proof}

\subsection{The bijections $\al_r$ and $\be_r$}\label{sec:al-be}

We are almost ready to define the key bijections $\al_r$ and $\be_r$. These are reminiscent of the bijections $\si_r$ and $\ta_r$ defined in~\cite{part1} for paths. An important difference, however, is that the image by $\be_r$ of a two-rowed array that encodes a path does not encode a path in general, so one cannot view $\be_r$ as a map on paths.

Let $\cd$ be a two-rowed array in either $\tra{(x,u]}{[y,v)}_{n\pm k}$ or $\tra{(x,v)}{[y,u]}_{n\pm k}$. We say that a crossing of $\cd$ at $c_i$ (resp.\ $d_i$) is {\em proper} if $c_i\notin\{u,v\}$ (resp.\ $d_i\notin\{u,v\}$).

For $r\ge1$, the map $\al_r$ applies to two-rowed arrays $\cd$ whose $r$th crossing is a proper upward crossing,
and it swaps the parts of the top and the bottom rows of the array to the right of this crossing. Schematically, if the $r$th crossing is at $c_i$, we have
\begin{center}
\begin{tikzpicture}[scale=0.5]
\draw (0.5,1) rectangle (5,2);
\draw (2.75,1.5) node {$\cdots$};
\draw (.5,1.5) node[right=-1] {$x$};
\draw (0,0) rectangle (4.5,1);
\draw (2.25,.5) node {$\cdots$};
\draw (0,.5) node[right=-1.5] {$y$};
\draw[fill=yellow] (4.5,1.5) circle (.4);
\draw (4.5,1.5) node {$c_i$};
\draw[pattern color=red,pattern=north west lines] (5,1) rectangle (9,2);
\draw (7.5,1.5) node {$\cdots$};
\draw[pattern color=blue,pattern=north east lines] (4.5,0) rectangle (7.5,1);
\draw (6.5,.5) node {$\cdots$};
\draw (4.5,.5) node[left=-1.5] {$d_{i}$};
\draw (5,1.5) node[right=-1.5] {$c_{i+1}$};
\draw (4.5,.5) node[right=-1.5] {$d_{i+1}$};
\draw (9,1.5) node[left=-1.5] {$u$};
\draw (7.5,.5) node[left=-1.5] {$v$};
\draw[<->] (10,1) -- node[above]{$\al_r$} (11,1);

\begin{scope}[shift={(12,0)}]
\draw (0.5,1) rectangle (5,2);
\draw (2.75,1.5) node {$\cdots$};
\draw (.5,1.5) node[right=-1.5] {$x$};
\draw (0,0) rectangle (4.5,1);
\draw (2.25,.5) node {$\cdots$};
\draw (0,.5) node[right=-1.5] {$y$};
\draw[fill=yellow] (4.5,1.5) circle (.4);
\draw (4.5,1.5) node {$c_i$};
\draw[pattern color=blue,pattern=north east lines] (5,1) rectangle (8,2);
\draw (7,1.5) node {$\cdots$};
\draw[pattern color=red,pattern=north west lines] (4.5,0) rectangle (8.5,1);
\draw (7,.5) node {$\cdots$};
\draw (4.5,.5) node[left=-1.5] {$d_{i}$};
\draw (5,1.5) node[right=-1.5] {$d_{i+1}$};
\draw (4.5,.5) node[right=-1.5] {$c_{i+1}$};
\draw (8,1.5) node[left=-1.5] {$v$};
\draw (8.5,.5) node[left=-1.5] {$u$};
\end{scope}
\end{tikzpicture}
\end{center}
The properness of the crossing guarantees that $c_{i+1}$ exists and that $c_i<c_{i+1}$.
Additionally, we have $c_i<d_{i+1}$ and $d_i<c_{i+1}$, so the rows of $\al_r(\cd)$ are increasing.
The two-rowed array $\al_r(\cd)$ has a crossing at $c_i$, since $d_i<c_i<c_{i+1}$, and this crossing is still proper. This is in fact the $r$th crossing of $\al_r(\cd)$, because the portion of the arrays to the left of $c_i$ is not affected by $\al_r$. It follows that $\al_r$ is an involution.

Similarly, the map $\be_r$ applies to two-rowed arrays $\cd$ whose $r$th crossing is a proper downward crossing, and it also swaps the top and the bottom rows of the array to the right of this crossing. Schematically, if the $r$th crossing is at $d_i$, we have
\begin{center}
\begin{tikzpicture}[scale=0.5]
\draw (0.5,1) rectangle (4.5,2);
\draw (2.25,1.5) node {$\cdots$};
\draw (.5,1.5) node[right=-1.5] {$x$};
\draw (0,0) rectangle (5,1);
\draw (2.5,.5) node {$\cdots$};
\draw (0,.5) node[right=-1.5] {$y$};
\draw[fill=yellow] (4.5,.5) circle (.4);
\draw (4.5,.5) node {$d_i$};
\draw[pattern color=red,pattern=north west lines] (4.5,1) rectangle (9,2);
\draw (6.75,1.5) node {$\cdots$};
\draw[pattern color=blue,pattern=north east lines] (5,0) rectangle (7.5,1);
\draw (6.5,.5) node {$\cdots$};
\draw (4.5,1.5) node[left=-1.5] {$c_{i-1}$};
\draw (4.5,1.5) node[right=-1.5] {$c_{i}$};
\draw (5,.5) node[right=-1.5] {$d_{i+1}$};
\draw (9,1.5) node[left=-1.5] {$u$};
\draw (7.5,.5) node[left=-1.5] {$v$};
\draw[<->] (10,1) -- node[above]{$\be_r$} (11,1);

\begin{scope}[shift={(12,0)}]
\draw (0.5,1) rectangle (4.5,2);
\draw (2.25,1.5) node {$\cdots$};
\draw (1,1.5) node {$x$};
\draw (0,0) rectangle (5,1);
\draw (2.5,.5) node {$\cdots$};
\draw (.5,.5) node {$y$};
\draw[fill=yellow] (4.5,.5) circle (.4);
\draw (4.5,.5) node {$d_i$};
\draw[pattern color=blue,pattern=north east lines] (4.5,1) rectangle (7,2);
\draw (6,1.5) node {$\cdots$};
\draw[pattern color=red,pattern=north west lines] (5,0) rectangle (9.5,1);
\draw (7.25,.5) node {$\cdots$};
\draw (4.5,1.5) node[left=-1.5] {$c_{i-1}$};
\draw (4.5,1.5) node[right=-1.5] {$d_{i+1}$};
\draw (5,.5) node[right=-1.5] {$c_{i}$};
\draw (7,1.5) node[left=-1.5] {$v$};
\draw (9.5,.5) node[left=-1.5] {$u$};
\end{scope}
\end{tikzpicture}
\end{center}
Again, the $r$th crossing of $\be_r(\cd)$ is still at $d_i$ and is a proper crossing, and the map $\be_r$ is an involution.

\begin{lemma}\label{lem:al,be}
Let $x,y,u,v,k\in\Z$, $n\ge0$ and $r\ge1$, satisfying that, if $u>v$, then $k\le0$, and if $u<v$, then $k\ge1$. The map $\al_r$ restricts to a bijection
$$\tra{(x,u]}{[y,v)}_{n\pm k}^{\ge r\uw}\overset{\al_r}{\longleftrightarrow}\tra{(x,v)}{[y,u]}_{n\mp k}^{\ge r\uw},$$
and the map $\be_r$ restricts to a bijection
$$\tra{(x,u]}{[y,v)}_{n\pm (k-1)}^{\ge r\dw}\overset{\be_r}{\longleftrightarrow}\tra{(x,v)}{[y,u]}_{n\mp k}^{\ge r\dw}.$$
Both $\al_r$ and $\be_r$ preserve the sum of the entries of the arrays.
\end{lemma}

\begin{proof}
The conditions on $k$, which depend on the relationship between $u$ and $v$, guarantee that the $r$th crossing of a two-rowed array in any of the four sets above is always proper, and so the maps $\al_r$ and $\be_r$ are  defined. Indeed, an improper upward crossing of $\cd\in\tra{(x,u]}{[y,v)}_{n\pm k}$ at $c_i$ could only occur if $u=c_i<d_{i+1}\le v$ and $k\le 0$.
Arrays $\cd\in\tra{(x,v)}{[y,u]}_{n\mp k}$ cannot have improper upward crossings, since $c_i=v$ is incompatible with $i\le n-|k|$.
An improper downward crossing of $\cd\in\tra{(x,u]}{[y,v)}_{n\pm (k-1)}$ at $d_i$ could only occur if $v=d_i<c_i\le u$ and $k-1\ge 0$. 
And an improper downward crossing of $\cd\in\tra{(x,v)}{[y,u]}_{n\mp k}$ at $d_i$ could only occur if $u=d_i<c_i\le v$ and $k\le 0$. 

Having already seen that $\al_r$ and $\be_r$ are involutions, it remains describe their images when restricted to two-rowed arrays whose rows have given lengths.
Given $\cd\in\tra{(x,u]}{[y,v)}_{n\pm k}^{\ge r\uw}$ whose $r$th crossing is at $c_i$, if we write $\cd$ as
\begin{center}
\begin{tikzpicture}[scale=.8]
\draw (.5,.8) node {$x$};
\draw (1.5,.8) node {$c_1$};
\draw (2.5,.8) node {$c_2$};
\draw (3.7,.8) node {$\cdots$};
\foreach \x in {1,2,3,4.4,5.4}{\draw (\x,.8) node {$<$};}
\draw[fill=yellow] (4.9,.8) circle (.25);
\draw (4.9,.8) node {$c_i$};
\foreach \x in {6.8,10.9} {\draw[red] (\x,.8) node {$<$};}
\draw[red] (6.1,.8) node {$c_{i+1}$};
\draw[red] (8.6,.8) node {$\cdots$};
\draw[red] (11.6,.8) node {$c_{n+k}$};
\draw[red] (12.3,.8) node {$\le$};
\draw[red] (12.8,.8) node {$u$};
\draw (0,0) node {$y$};
\draw (.5,0) node {$\le$};
\draw (1,0) node {$d_1$};
\draw (2,0) node {$d_2$};
\draw (3.2,0) node {$\cdots$};
\foreach \x in {1.5,2.5,3.9,4.9} {\draw (\x,0) node {$<$};}
\draw (4.4,0) node {$d_i$};
\foreach \x in {6.3,7.8,9.2,10.6} {\draw[blue] (\x,0) node {$<$};}
\draw[blue] (5.6,0) node {$d_{i+1}$};
\draw[blue] (7,0) node {$d_{i+2}$};
\draw[blue] (8.6,0) node {$\cdots$};
\draw[blue] (9.9,0) node {$d_{n-k}$};
\draw[blue] (11.1,0) node {$v$};
\draw (4.65,.4) node {$\lessup$};
\draw (5.25,.4) node {$\lessdown$};
\end{tikzpicture},
\end{center}
then $\al_r(\cd)$ is the two-rowed array
\begin{center}
\begin{tikzpicture}[scale=.8]
\draw (.5,.8) node {$x$};
\draw (1.5,.8) node {$c_1$};
\draw (2.5,.8) node {$c_2$};
\draw (3.7,.8) node {$\cdots$};
\foreach \x in {1,2,3,4.4,5.4}{\draw (\x,.8) node {$<$};}
\draw[fill=yellow] (4.9,.8) circle (.25);
\draw (4.9,.8) node {$c_i$};
\begin{scope}[shift={(-.5,-.85)}]
\foreach \x in {6.8,10.9} {\draw[red] (\x,.8) node {$<$};}
\draw[red] (6.1,.8) node {$c_{i+1}$};
\draw[red] (8.6,.8) node {$\cdots$};
\draw[red] (11.6,.8) node {$c_{n+k}$};
\draw[red] (12.3,.8) node {$\le$};
\draw[red] (12.8,.8) node {$u$};
\end{scope}
\draw (0,0) node {$y$};
\draw (.5,0) node {$\le$};
\draw (1,0) node {$d_1$};
\draw (2,0) node {$d_2$};
\draw (3.2,0) node {$\cdots$};
\foreach \x in {1.5,2.5,3.9,4.9} {\draw (\x,0) node {$<$};}
\draw (4.4,0) node {$d_i$};
\begin{scope}[shift={(.5,.85)}]
\foreach \x in {6.3,7.8,9.2,10.6} {\draw[blue] (\x,0) node {$<$};}
\draw[blue] (5.6,0) node {$d_{i+1}$};
\draw[blue] (7,0) node {$d_{i+2}$};
\draw[blue] (8.6,0) node {$\cdots$};
\draw[blue] (9.9,0) node {$d_{n-k}$};
\draw[blue] (11.1,0) node {$v$};
\end{scope}
\draw (4.65,.4) node {$\lessup$};
\draw (5.25,.4) node {$\lessdown$};
\end{tikzpicture},
\end{center}
which belongs to $\tra{(x,v)}{[y,u]}_{n\mp k}^{\ge r\uw}$, and every element of this set is obtained in this way.

Similarly, given $\cd\in\tra{(x,u]}{[y,v)}_{n\pm (k-1)}^{\ge r\dw}$ whose $r$th crossing is at $d_i$, if we write $\cd$ as
\begin{center}
\begin{tikzpicture}[scale=.8]
\draw (.5,.8) node {$x$};
\draw (1.5,.8) node {$c_1$};
\draw (2.5,.8) node {$c_2$};
\draw (3.7,.8) node {$\cdots$};
\foreach \x in {1,2,3,4.4,5.8}{\draw (\x,.8) node {$<$};}
\draw (5.1,.8) node {$c_{i-1}$};
\draw[red] (6.3,.8) node {$c_{i}$};
\foreach \x in {6.8,8.2,10.9} {\draw[red] (\x,.8) node {$<$};}
\draw[red] (7.5,.8) node {$c_{i+1}$};
\draw[red] (9.4,.8) node {$\cdots$};
\draw[red] (11.8,.8) node {$c_{n+k-1}$};
\draw[red] (12.7,.8) node {$\le$};
\draw[red] (13.2,.8) node {$u$};
\draw (0,0) node {$y$};
\draw (.5,0) node {$\le$};
\draw (1,0) node {$d_1$};
\draw (2,0) node {$d_2$};
\draw (3.7,0) node {$\cdots$};
\foreach \x in {1.5,2.5,5.3,6.3} {\draw (\x,0) node {$<$};}
\draw[fill=yellow] (5.8,.0) circle (.25);
\draw (5.8,0) node {$d_i$};
\draw[blue] (7,0) node {$d_{i+1}$};
\foreach \x in {7.7,9.2,11} {\draw[blue] (\x,0) node {$<$};}
\draw[blue] (8.6,0) node {$\cdots$};
\draw[blue] (10.1,0) node {$d_{n-k+1}$};
\draw[blue] (11.5,0) node {$v$};
\draw (5.45,.4) node {$\lessdown$};
\draw (6.05,.4) node {$\lessup$};
\end{tikzpicture},
\end{center}
then $\be_r(\cd)$ is the two-rowed array
\begin{center}
\begin{tikzpicture}[scale=.8]
\draw (.5,.8) node {$x$};
\draw (1.5,.8) node {$c_1$};
\draw (2.5,.8) node {$c_2$};
\draw (3.7,.8) node {$\cdots$};
\foreach \x in {1,2,3,4.4,5.8}{\draw (\x,.8) node {$<$};}
\draw (5.1,.8) node {$c_{i-1}$};
\begin{scope}[shift={(.5,-.85)}]
\draw[red] (6.3,.8) node {$c_{i}$};
\foreach \x in {6.8,8.2,10.9} {\draw[red] (\x,.8) node {$<$};}
\draw[red] (7.5,.8) node {$c_{i+1}$};
\draw[red] (9.4,.8) node {$\cdots$};
\draw[red] (11.8,.8) node {$c_{n+k-1}$};
\draw[red] (12.7,.8) node {$\le$};
\draw[red] (13.2,.8) node {$u$};
\end{scope}
\draw (0,0) node {$y$};
\draw (.5,0) node {$\le$};
\draw (1,0) node {$d_1$};
\draw (2,0) node {$d_2$};
\draw (3.7,0) node {$\cdots$};
\foreach \x in {1.5,2.5,5.3,6.3} {\draw (\x,0) node {$<$};}
\draw[fill=yellow] (5.8,.0) circle (.25);
\draw (5.8,0) node {$d_i$};
\begin{scope}[shift={(-.5,.85)}]
\draw[blue] (7,0) node {$d_{i+1}$};
\foreach \x in {7.7,9.2,11} {\draw[blue] (\x,0) node {$<$};}
\draw[blue] (8.6,0) node {$\cdots$};
\draw[blue] (10.1,0) node {$d_{n-k+1}$};
\draw[blue] (11.5,0) node {$v$};
\end{scope}
\draw (5.45,.4) node {$\lessdown$};
\draw (6.05,.4) node {$\lessup$};
\end{tikzpicture},
\end{center}
which is an arbitrary element in $\tra{(x,v)}{[y,u]}_{n\mp k}^{\ge r\dw}$.

It is clear by construction that both $\al_r$ and $\be_r$ preserve the sum of the entries of the arrays.
\end{proof}

\subsection{Proof of Theorem~\ref{thm:line:des}}

For $a,b,r\ge0$ and $\ell\in\Z$, we interpret elements of $\G^{\ge r,\ell}_{a,b}$ as paths in $\P_{(x,y)\to(u,v)}^{\ge r}$, 
where $(x,y)=(\ell,0)$ and $(u,v)=(b+\ell,a)$, using the transformation~\eqref{bij:GtoP}. 
For any $n\ge0$, the subset of paths having $n$ descents is in bijection with the set $\tra{(x,u]}{[y,v)}_{n}^{\ge r}$, using the encoding~\eqref{bij:Pto2ra:cro}.

The proof is divided into nine cases according to whether the paths start below ($0<\ell$, equivalently $x>y$), on ($0=\ell$, equivalently $x=y$), or above ($0>\ell$, equivalently $x<y$) the line being crossed, and whether they end below ($\ell>a-b$, equivalently $u>v$), on ($\ell=a-b$, equivalently $u=v$), or above ($\ell<a-b$, equivalently $u<v$) this line. 
In each case, we determine $G^{\ge r,\ell}_{a,b}(t,q)$ by first using Equation~\eqref{eq:Gto2ra} to rewrite it in terms of two-rowed arrays, then repeatedly applying the maps from Lemma~\ref{lem:al,be} to construct bijections between $\tra{(x,u]}{[y,v)}_{n}^{\ge r}$ and certain sets of two-rowed arrays with no requirement on the number of crossings, and finally using Lemma~\ref{lem:sum}.
The cases are labeled as in~\cite{part1} for consistency, but we will prove them in a slightly different order.

\begin{list1}

\item {\bf Case I:} $0<\ell<a-b$, equivalently $x>y$ and $u<v$.
By Equation~\eqref{eq:x>y,u<v1} with $s=0$, $$\tra{(x,u]}{[y,v)}_{n}^{\ge 2m+1}=\tra{(x,u]}{[y,v)}_{n}^{\ge 2m},$$
and so $G^{\ge 2m+1,\ell}_{a,b}(t,q)=G^{\ge 2m,\ell}_{a,b}(t,q)$.
Using Lemmas~\ref{lem:even=odd}(a) and~\ref{lem:al,be}, noting that the condition $k\ge1$ in the latter holds at each step, we construct
a composition of bijections $\al_1\circ\be_2\circ\dots\circ\al_{2m-1}\circ\be_{2m}$:
\begin{align}\nonumber
\tra{(x,u]}{[y,v)}_{n}^{\ge 2m}=\tra{(x,u]}{[y,v)}_{n}^{\ge 2m\dw}&\overset{\be_{2m}}{\longrightarrow}\tra{(x,v)}{[y,u]}_{n\mp 1}^{\ge 2m\dw}
=\tra{(x,v)}{[y,u]}_{n\mp 1}^{\ge 2m-1\uw}
\overset{\al_{2m-1}}{\longrightarrow}
\tra{(x,u]}{[y,v)}_{n\pm 1}^{\ge 2m-1\uw}=\tra{(x,u]}{[y,v)}_{n\pm 1}^{\ge 2m-2\dw}\\
&\overset{\be_{2m-2}}{\longrightarrow}\cdots\overset{\al_{1}}{\longrightarrow}
\tra{(x,u]}{[y,v)}_{n\pm m}^{\ge 1\uw}=\tra{(x,u]}{[y,v)}_{n\pm m}^{\ge 0\dw}=\tra{(x,u]}{[y,v)}_{n\pm m}.
\label{eq:for_figI}
\end{align}
See Figure~\ref{fig:CaseI} for an example. Since these bijections preserve the sum of the entries of the two-rowed arrays, Equation~\eqref{eq:sum1} gives
\begin{multline}\sum_{\cd\in\tra{(x,u]}{[y,v)}_{n}^{\ge 2m}} q^{\su\bfc+\su\bfd-n(x+y)}
=\sum_{\cd\in\tra{(x,u]}{[y,v)}_{n\pm m}} q^{\su\bfc+\su\bfd-n(x+y)}\\
=q^{n^2+m(m+x-y+1)}\qbin{u-x}{n+m}\qbin{v-y}{n-m}
=q^{n^2+m(m+\ell+1)}\qbin{a}{n-m}\qbin{b}{n+m}.\label{eq:caseI}\end{multline}
Using Equation~\eqref{eq:Gto2ra}, this proves Equation~\eqref{eq:0<l<a-b:des}.

\begin{figure}[htb]
\centering\medskip
$\begin{array}{ccccc}
\tra{(1,7]}{[0,8)}_{4}^{\ge 2}=\tra{(1,7]}{[0,8)}_{4}^{\ge 2\dw}&
\overset{\be_{2}}{\longrightarrow}&
\tra{(1,8)}{[0,7]}_{4\mp 1}^{\ge 2\dw}=\tra{(1,8)}{[0,7]}_{4\mp 1}^{\ge 1\uw}&
\overset{\al_{1}}{\longrightarrow}&
\tra{(1,7]}{[0,8)}_{4\pm 1}^{\ge 1\uw}=\tra{(1,7]}{[0,8)}_{4\pm 1}\medskip\\
\begin{tikzpicture}[scale=0.7]
\uwcrossingc{2}{olive}
\dwcrossingfillc{4}{violet}
\uwcrossingc{4}{orange}
\drawarray{{1,2,3,4,6,7}}{{0,0,1,4,5,8}}
\ineq{5}{5}
\end{tikzpicture}
&&
\begin{tikzpicture}[scale=0.7]
\uwcrossingfillc{2}{olive}
\dwcrossingc{4}{violet}
\drawarray{{1,2,3,4,8}}{{0,0,1,4,5,6,7}}
\ineqx{4}{6}
\end{tikzpicture}
&&
\begin{tikzpicture}[scale=0.7]
\uwcrossingc{2}{olive}
\drawarray{{1,2,3,4,5,6,7}}{{0,0,1,4,8}}
\ineq{6}{4}
\end{tikzpicture}
\end{array}$
\caption{An example of the bijection~\eqref{eq:for_figI}, where $(x,y)=(1,0)$, $(u,v)=(7,8)$, $m=1$ and $n=4$.}
\label{fig:CaseI}
\end{figure}

\item {\bf Case II:} $0>\ell>a-b$, equivalently $x<y$ and $u>v$. Similarly to Case~I, the equality $G^{\ge 2m+1,\ell}_{a,b}(t,q)=G^{\ge 2m,\ell}_{a,b}(t,q)$ follows now from Equation~\eqref{eq:x<y,u>v1} with $s=0$. 
Again Lemmas~\ref{lem:even=odd}(d) and~\ref{lem:al,be}, noting that the condition $k\le0$ holds at each step,
allow us to build a sequence of bijections $\be_1\circ\al_2\circ\dots\circ\be_{2m-1}\circ\al_{2m}$:
\begin{align*}\tra{(x,u]}{[y,v)}_{n}^{\ge 2m}=\tra{(x,u]}{[y,v)}_{n}^{\ge 2m\uw}&\overset{\al_{2m}}{\longrightarrow}\tra{(x,v)}{[y,u]}_{n}^{\ge 2m\uw}
=\tra{(x,v)}{[y,u]}_{n}^{\ge 2m-1\dw}
\overset{\be_{2m-1}}{\longrightarrow}
\tra{(x,u]}{[y,v)}_{n\mp 1}^{\ge 2m-1\dw}=\tra{(x,u]}{[y,v)}_{n\mp 1}^{\ge 2m-2\uw}\\
&\overset{\al_{2m-2}}{\longrightarrow}\cdots\overset{\be_{1}}{\longrightarrow}
\tra{(x,u]}{[y,v)}_{n\mp m}^{\ge 1\dw}=\tra{(x,u]}{[y,v)}_{n\mp m}^{\ge 0\uw}=\tra{(x,u]}{[y,v)}_{n\mp m}.
\end{align*}
Then, by Equation~\eqref{eq:sum1},
\begin{multline}\sum_{\cd\in\tra{(x,u]}{[y,v)}_{n}^{\ge 2m}} q^{\su\bfc+\su\bfd-n(x+y)}
=\sum_{\cd\in\tra{(x,u]}{[y,v)}_{n\mp m}} q^{\su\bfc+\su\bfd-n(x+y)}\\
=q^{n^2-m(-m+x-y+1)}\qbin{u-x}{n-m}\qbin{v-y}{n+m}
=q^{n^2+m(m-\ell-1)}\qbin{a}{n+m}\qbin{b}{n-m},\label{eq:caseII}
\end{multline}
proving Equation~\eqref{eq:0>l>a-b:des}.

\item {\bf Case III:} $0>\ell<a-b$, equivalently $x<y$ and $u<v$.
The equality $G^{\ge 2m+2,\ell}_{a,b}(t,q)=G^{\ge 2m+1,\ell}_{a,b}(t,q)$ follows now from Equation~\eqref{eq:x<y,u<v1} with $s=0$. 
Lemmas~\ref{lem:even=odd}(c) and~\ref{lem:al,be} produce a sequence of bijections $\be_1\circ\al_2\circ\dots\circ\be_{2m+1}$:
\begin{align*}\tra{(x,u]}{[y,v)}_{n}^{\ge 2m+1}=\tra{(x,u]}{[y,v)}_{n}^{\ge 2m+1\dw}&
\overset{\be_{2m+1}}{\longrightarrow}
\tra{(x,v)}{[y,u]}_{n\mp 1}^{\ge 2m+1\dw}=\tra{(x,v)}{[y,u]}_{n\mp 1}^{\ge 2m\uw}
\overset{\al_{2m}}{\longrightarrow}
\tra{(x,u]}{[y,v)}_{n\pm 1}^{\ge 2m\uw}=\tra{(x,u]}{[y,v)}_{n\pm 1}^{\ge 2m-1\dw}\\
&\overset{\be_{2m-1}}{\longrightarrow}\cdots\overset{\be_{1}}{\longrightarrow}
\tra{(x,v)}{[y,u]}_{n\mp (m+1)}^{\ge 1\dw}=\tra{(x,v)}{[y,u]}_{n\mp (m+1)}^{\ge 0\uw}=\tra{(x,v)}{[y,u]}_{n\mp (m+1)}.
\end{align*}
By Equation~\eqref{eq:sum2},
\begin{multline}\sum_{\cd\in\tra{(x,u]}{[y,v)}_{n}^{\ge 2m+1}} q^{\su\bfc+\su\bfd-n(x+y)}
=\sum_{\cd\in\tra{(x,v)}{[y,u]}_{n\mp (m+1)}} q^{\su\bfc+\su\bfd-n(x+y)}\\
=q^{n^2-(m+1)(-m+x-y)}\qbin{v-x-1}{n-m-1}\qbin{u-y+1}{n+m+1}
=q^{n^2+(m+1)(m-\ell)}\qbin{a-\ell-1}{n-m-1}\qbin{b+\ell+1}{n+m+1},\label{eq:caseIII}\end{multline}
proving Equation~\eqref{eq:0>l<a-b:des}.

\item {\bf Case IV:} $0<\ell>a-b$, equivalently $x>y$ and $u>v$.
Here $G^{\ge 2m+2,\ell}_{a,b}(t,q)=G^{\ge 2m+1,\ell}_{a,b}(t,q)$ because of Equation~\eqref{eq:x>y,u>v1} with $s=0$. 
Lemmas~\ref{lem:even=odd}(b) and~\ref{lem:al,be} give a sequence of bijections $\al_1\circ\be_2\circ\dots\circ\al_{2m+1}$:
\begin{align*}\tra{(x,u]}{[y,v)}_{n}^{\ge 2m+1}=\tra{(x,u]}{[y,v)}_{n}^{\ge 2m+1\uw}&
\overset{\al_{2m+1}}{\longrightarrow}
\tra{(x,v)}{[y,u]}_{n}^{\ge 2m+1\uw}=\tra{(x,v)}{[y,u]}_{n}^{\ge 2m\dw}
\overset{\be_{2m}}{\longrightarrow}
\tra{(x,u]}{[y,v)}_{n\mp 1}^{\ge 2m\dw}=\tra{(x,u]}{[y,v)}_{n\mp 1}^{\ge 2m-1\uw}\\
&\overset{\al_{2m-1}}{\longrightarrow}\cdots\overset{\al_{1}}{\longrightarrow}
\tra{(x,v)}{[y,u]}_{n\pm m}^{\ge 1\uw}=\tra{(x,v)}{[y,u]}_{n\pm m}^{\ge 0\dw}=\tra{(x,v)}{[y,u]}_{n\pm m}.
\end{align*}
By Equation~\eqref{eq:sum2},
\begin{multline}\sum_{\cd\in\tra{(x,u]}{[y,v)}_{n}^{\ge 2m+1}} q^{\su\bfc+\su\bfd-n(x+y)}
=\sum_{\cd\in\tra{(x,v)}{[y,u]}_{n\pm m}} q^{\su\bfc+\su\bfd-n(x+y)}\\
=q^{n^2+m(m+x-y+1)}\qbin{v-x-1}{n+m}\qbin{u-y+1}{n-m}
=q^{n^2+m(m+\ell+1)}\qbin{a-\ell-1}{n+m}\qbin{b+\ell+1}{n-m},\label{eq:caseIV}\end{multline}
proving Equation~\eqref{eq:0<l>a-b:des}.

\item {\bf Case VII:} $0<\ell=a-b$, equivalently $x>y$ and $u=v$. In this case, the parity of the total number of crossings is not
forced by the endpoints, so we consider the cases $r=2m$ and $r=2m+1$ separately. The case $r=2m$ is proved like Case~I, 
constructing a sequence of bijections $\al_1\circ\be_2\circ\dots\circ\al_{2m-1}\circ\be_{2m}$:
\begin{equation}\label{eq:bijVIIeven}
\tra{(x,u]}{[y,v)}_{n}^{\ge 2m}=\tra{(x,u]}{[y,v)}_{n}^{\ge 2m\dw}
\longrightarrow
\tra{(x,u]}{[y,v)}_{n\pm m}^{\ge 1\uw}=\tra{(x,u]}{[y,v)}_{n\pm m}^{\ge 0\dw}=\tra{(x,u]}{[y,v)}_{n\pm m},
\end{equation}
where we use Lemma~\ref{lem:first_crossing} for the left equality, and Lemmas~\ref{lem:even=odd}(e) and~\ref{lem:al,be} to compose the bijections. Equation~\eqref{eq:0<l=a-b:des-even} now follows using Equation~\eqref{eq:caseI} again.

The case $r=2m+1$ is proved like Case~IV, 
constructing a sequence of bijections $\al_1\circ\be_2\circ\dots\circ\al_{2m+1}$:
\begin{equation}\label{eq:bijVIIodd}
\tra{(x,u]}{[y,v)}_{n}^{\ge 2m+1}=\tra{(x,u]}{[y,v)}_{n}^{\ge 2m+1\uw}
\longrightarrow
\tra{(x,v)}{[y,u]}_{n\pm m}^{\ge 1\uw}=\tra{(x,v)}{[y,u]}_{n\pm m}^{\ge 0\dw}=\tra{(x,v)}{[y,u]}_{n\pm m}.
\end{equation}
Now we use Equation~\eqref{eq:caseIV} and the fact that $\ell=a-b$ to prove Equation~\eqref{eq:0<l=a-b:des-odd}.

\item {\bf Case VIII:} $0>\ell=a-b$, equivalently $x<y$ and $u=v$. This case is analogous to Case~VII. When $r=2m$, we use the same sequence bijections as in Case~II,
\begin{equation}\label{eq:bijVIIIeven}
\be_1\circ\al_2\circ\dots\circ\be_{2m-1}\circ\al_{2m}:\tra{(x,u]}{[y,v)}_{n}^{\ge 2m}\longrightarrow
\tra{(x,u]}{[y,v)}_{n\mp m},
\end{equation}
using Lemma~\ref{lem:even=odd}(f).
Equation~\eqref{eq:0>l=a-b:des-even} now follows from Equation~\eqref{eq:caseII}.

When $r=2m+1$, we use the same sequence of bijections as in Case~III,
\begin{equation}\label{eq:bijVIIIodd}\be_1\circ\al_2\circ\dots\circ\be_{2m+1}:
\tra{(x,u]}{[y,v)}_{n}^{\ge 2m+1}\longrightarrow\tra{(x,v)}{[y,u]}_{n\mp (m+1)}.
\end{equation}
Equation~\eqref{eq:0>l=a-b:des-odd} follows from Equation~\eqref{eq:caseIII} after the substitution $\ell=a-b$.

\item {\bf Case V:} $0=\ell<a-b$, equivalently $x=y$ and $u<v$.
We will reduce this case to Case~VIII by applying an involution $\nu$ on two-rowed arrays that changes the sign of each entry, reverses each row (so that the negated entries increase from left to right), and swaps the top and the bottom rows. The map $\nu$ restricts to bijections
\begin{equation}\label{def:nu}\tra{(x,u]}{[y,v)}_{n\pm k}\overset{\nu}{\longleftrightarrow}\tra{(-v,-y]}{[-u,-x)}_{n\mp k},\qquad
\tra{(x,v)}{[y,u]}_{n\pm k}\overset{\nu}{\longleftrightarrow}\tra{[-u,-y]}{(-v,-x)}_{n\mp k}\end{equation}
for any $k\in\Z$. Additionally, in the case $k=0$, it restricts to a bijection
\begin{equation}\label{eq:nu}\tra{(x,u]}{[y,v)}_{n}^{\ge r}\overset{\nu}{\longleftrightarrow}\tra{(-v,-y]}{[-u,-x)}_{n}^{\ge r},\end{equation}
since it preserves the number of crossings; specifically, upward crossings turn into downward crossings, and vice-versa. Indeed, the two-rowed array
$$\begin{array}{cccccccccccc}
&x&<&c_1&<&c_2&<&\dots&<&c_{n}&\le&u\\
y&\le &d_1&<&d_2&<&\dots&<&d_{n}&<&v&
\end{array}$$
is mapped by $\nu$  to
$$\begin{array}{cccccccccccc}
&-v&<&-d_n&<&-d_{n-1}&<&\dots&<&-d_1&\le&-y\\
-u&\le &-c_n&<&-c_{n-1}&<&\dots&<&-c_1&<&-x&
\end{array}.$$
Thus, the first array has an upward crossing at $c_i$ if and only if the second one has a downward crossing at $-c_i$, since condition~\eqref{eq:upward} is equivalent to $-d_{i+1}<-c_i<-d_i$, and similarly for the other type of crossing. In terms of the corresponding lattice paths given by the encoding~\eqref{bij:Pto2ra:cro}, the involution $\nu$ translates to a reflection along the line $x+y=0$.

The conditions $x=y$ and $u<v$ are equivalent to $-v<-u$ and $-y=-x$, so we can apply the bijections from Case~VIII to the set on the right-hand side of~\eqref{eq:nu}. When $r=2m$, Equation~\eqref{eq:bijVIIIeven} gives a bijection 
$$\be_1\circ\al_2\circ\dots\circ\be_{2m-1}\circ\al_{2m}:\tra{(-v,-y]}{[-u,-x)}_{n}^{\ge 2m}\longrightarrow
\tra{(-v,-y]}{[-u,-x)}_{n\mp m}.$$
Conjugating by $\nu$, we get a bijection
$$\nu\circ\be_1\circ\al_2\circ\dots\circ\be_{2m-1}\circ\al_{2m}\circ\nu:\tra{(x,u]}{[y,v)}_{n}^{\ge 2m}\longrightarrow
\tra{(x,u]}{[y,v)}_{n\pm m}$$
that preserves the sum of the entries. Using Equation~\eqref{eq:caseI} with $\ell=0$, we deduce Equation~\eqref{eq:0=l<a-b:des-even}.

When $r=2m+1$, Equation~\eqref{eq:bijVIIIodd} gives a bijection 
$$\be_1\circ\al_2\circ\dots\circ\be_{2m+1}:
\tra{(-v,-y]}{[-u,-x)}_{n}^{\ge 2m+1}\longrightarrow\tra{(-v,-x)}{[-u,-y]}_{n\mp (m+1)},$$
and conjugating by $\nu$ we get
$$\nu\circ\be_1\circ\al_2\circ\dots\circ\be_{2m+1}\circ\nu:\tra{(x,u]}{[u,v)}_{n}^{\ge 2m}\longrightarrow
\tra{[y,u]}{(x,v)}_{n\pm (m+1)}.$$
Swapping the top and bottom rows and using Equation~\eqref{eq:caseIII} with $\ell=0$, we deduce Equation~\eqref{eq:0=l<a-b:des-odd} .

\item {\bf Case VI:} $0=\ell>a-b$, equivalently $x=y$ and $u>v$. By applying the map $\nu$, this case reduces to Case~VII, since the conditions $x=y$ and $u>v$ are equivalent to $-v>-u$ and $-y=-x$.
When $r=2m$, conjugating the bijection~\eqref{eq:bijVIIeven} with $\nu$ gives a bijection 
$$\nu\circ\al_1\circ\be_2\circ\dots\circ\al_{2m-1}\circ\be_{2m}\circ\nu:\tra{(x,u]}{[y,v)}_{n}^{\ge 2m}\longrightarrow
\tra{(x,u]}{[y,v)}_{n\mp m}.$$
Using Equation~\eqref{eq:caseII} with $\ell=0$, we deduce Equation~\eqref{eq:0=l>a-b:des-even}.

When $r=2m+1$, conjugating the bijection~\eqref{eq:bijVIIodd} with $\nu$ gives a bijection 
$$\nu\circ\al_1\circ\be_2\circ\dots\circ\al_{2m+1}\circ\nu:
\tra{(x,u]}{[y,v)}_{n}^{\ge 2m+1}
\longrightarrow
\tra{[y,u]}{(x,v)}_{n\mp m}.$$
Swapping the top and bottom rows and using Equation~\eqref{eq:caseIV} with $\ell=0$,
we deduce Equation~\eqref{eq:0=l>a-b:des-odd}.

\item {\bf Case IX:} $0=\ell=a-b$, equivalently $x=y$ and $u=v$.
We consider two cases according to the first step of the path.
Via the bijection~\eqref{bij:Pto2ra:cro}, paths in the left-hand side starting with an $N$ are encoded by two-rowed arrays $\cd\in\tra{(x,u]}{[y,v)}_{n}^{\ge r}$ with $y<d_1$; equivalently, by 
$\cd\in\tra{(x,u]}{[y+1,v)}_{n}^{\ge r}$. Note that replacing the lower bound $y$ with $y+1$ does not affect the number of crossings of the array, since $\cd$ cannot have a crossing at $c_0$ in either case.
Since $x<y+1$, the conditions in Case~VIII hold with $y+1$ playing the role of $y$. Equation~\eqref{eq:bijVIIIeven} gives a bijection
$$\be_1\circ\al_2\circ\dots\circ\be_{2m-1}\circ\al_{2m}:\tra{(x,u]}{[y+1,v)}_{n}^{\ge 2m}\longrightarrow\tra{(x,u]}{[y+1,v)}_{n\mp m},$$
and using Equation~\eqref{eq:sum1}, it follows that
\begin{multline}\sum_{\cd\in\tra{(x,u]}{[y+1,v)}_{n}^{\ge 2m}} q^{\su\bfc+\su\bfd-n(x+y)}
=q^n\sum_{\cd\in\tra{(x,u]}{[y+1,v)}_{n\mp m}} q^{\su\bfc+\su\bfd-n(x+y+1)}\\
=q^n q^{n^2-m(-m+x-y)}\qbin{u-x}{n-m}\qbin{v-y-1}{n+m}
=q^{n^2+n+m^2}\qbin{a}{n-m}\qbin{a-1}{n+m}.\label{eq:Neven}
\end{multline}
Similarly, Equation~\eqref{eq:bijVIIIodd} gives a bijection
$$\be_1\circ\al_2\circ\dots\circ\be_{2m+1}:\tra{(x,u]}{[y+1,v)}_{n}^{\ge 2m+1}\longrightarrow\tra{(x,v)}{[y+1,u]}_{n\mp (m+1)},$$
and Equation~\eqref{eq:sum2} implies that
\begin{multline}\sum_{\cd\in\tra{(x,u]}{[y+1,v)}_{n}^{\ge 2m+1}} q^{\su\bfc+\su\bfd-n(x+y)}
=q^n\sum_{\cd\in\tra{(x,v)}{[y+1,u]}_{n\mp (m+1)}} q^{\su\bfc+\su\bfd-n(x+y+1)}\\
=q^n q^{n^2-(m+1)(-m-1+x-y)}\qbin{v-x-1}{n-m-1}\qbin{u-y}{n+m+1}
=q^{n^2+n+(m+1)^2}\qbin{a-1}{n-m-1}\qbin{a}{n+m+1}.\label{eq:Nodd}\end{multline}

On the other hand, paths in the left-hand side of~\eqref{bij:Pto2ra:cro} starting with an $E$ are encoded by two-rowed arrays $\cd\in\tra{(x,u]}{[y,v)}_{n}^{\ge r}$ with $y=d_1$. Let us use the notation
$\tra{(x,u]_n}{\llbracket y,v)_n}^{\ge r}$ for such arrays, where the double bracket indicates that the first element in the bottom row is forced to equal its lower bound.
By Lemma~\ref{lem:even=odd}(g), we can use the same bijections as in Case~VII, noting that the condition $y=d_1$ is preserved when applying the maps from Lemma~\ref{lem:al,be}. For $r=2m$, we get a bijection
$$\al_1\circ\be_2\circ\dots\circ\al_{2m-1}\circ\be_{2m}:\tra{(x,u]}{\llbracket y,v)}_{n}^{\ge 2m}\longrightarrow
\tra{(x,u]}{\llbracket y,v)}_{n\pm m},
$$
and Equation~\eqref{eq:sum1} implies that
\begin{align}\nonumber \sum_{\cd\in\tra{(x,u]}{\llbracket y,v)}_{n}^{\ge 2m}} & q^{\su\bfc+\su\bfd-n(x+y)}
=\sum_{\cd\in\tra{(x,u]}{\llbracket y,v)}_{n\pm m}} q^{\su\bfc+\su\bfd-n(x+y)}\\ \nonumber
&=\sum_{\cd\in\tra{(x,u]}{[y,v)}_{n\pm m}} q^{\su\bfc+\su\bfd-n(x+y)}-
q^n\sum_{\cd\in\tra{(x,u]}{[y+1,v)}_{n\pm m}} q^{\su\bfc+\su\bfd-n(x+y+1)}\\ \nonumber
&=q^{n^2+m(m+x-y+1)}\qbin{u-x}{n+m}\qbin{v-y}{n-m}
-q^n q^{n^2+m(m+x-y)}\qbin{u-x}{n+m}\qbin{v-y-1}{n-m}\\ \nonumber
&=q^{n^2+m(m+1)}\qbin{a}{n+m}\qbin{a}{n-m}
-q^n q^{n^2+m^2}\qbin{a}{n+m}\qbin{a-1}{n-m}\\ \nonumber
&=q^{n^2+m(m+1)}\qbin{a}{n+m}\left(\qbin{a}{n-m}-q^{n-m}\qbin{a-1}{n-m}\right)\\
&=q^{n^2+m(m+1)}\qbin{a}{n+m}\qbin{a-1}{n-m-1}.
\label{eq:Eeven}
\end{align}
Similarly, for $r=2m+1$, we get a bijection
\begin{equation}\label{eq:for_figIX}
\al_1\circ\be_2\circ\dots\circ\al_{2m+1}:
\tra{(x,u]}{\llbracket y,v)}_{n}^{\ge 2m+1}
\longrightarrow
\tra{(x,v)}{\llbracket y,u]}_{n\pm m}
\end{equation}
(see the example in Figure~\ref{fig:CaseIX}), and Equation~\eqref{eq:sum2} implies that
\begin{align}\nonumber &\sum_{\cd\in\tra{(x,u]}{\llbracket y,v)}_{n}^{\ge 2m+1}} q^{\su\bfc+\su\bfd-n(x+y)}
=\sum_{\cd\in\tra{(x,v)}{\llbracket y,u]}_{n\pm m}} q^{\su\bfc+\su\bfd-n(x+y)}\\ \nonumber
&=\sum_{\cd\in\tra{(x,v)}{[y,u]}_{n\pm m}} q^{\su\bfc+\su\bfd-n(x+y)}-
q^n\sum_{\cd\in\tra{(x,v)}{[y+1,u]}_{n\pm m}} q^{\su\bfc+\su\bfd-n(x+y+1)}\\ \nonumber
&=q^{n^2+m(m+x-y+1)}\qbin{v-x-1}{n+m}\qbin{u-y+1}{n-m}
-q^n q^{n^2+m(m+x-y)}\qbin{v-x-1}{n+m}\qbin{u-y}{n-m}\\ \nonumber
&=q^{n^2+m(m+1)}\qbin{a-1}{n+m}\qbin{a+1}{n-m}
-q^n q^{n^2+m^2}\qbin{a-1}{n+m}\qbin{a}{n-m}\\
\nonumber &=q^{n^2+m(m+1)}\qbin{a-1}{n+m}\left(\qbin{a+1}{n-m}
-q^{n-m}\qbin{a}{n-m}\right)\\
&=q^{n^2+m(m+1)}\qbin{a-1}{n+m}\qbin{a}{n-m-1}.
\label{eq:Eodd}
\end{align}

\begin{figure}[htb]
\centering\medskip
$\begin{array}{ccccccc}
\tra{(0,7]}{\llbracket 0,7)}_{3}^{\ge 3}=\tra{(0,7]}{\llbracket 0,7)}_{3}^{\ge 3\uw}&
\overset{\al_{3}}{\rightarrow}&
\tra{(0,7)}{\llbracket 0,7]}_{3}^{\ge 3\uw}=\tra{(0,7)}{\llbracket 0,7]}_{3}^{\ge 2\dw}&
\overset{\al_{1}}{\rightarrow}&
\tra{(0,7]}{\llbracket 0,7)}_{3\mp 1}^{\ge 2\dw}=\tra{(0,7]}{\llbracket 0,7)}_{3\mp 1}^{\ge 1\uw}&
\overset{\al_{1}}{\rightarrow}&
\tra{(0,7)}{\llbracket 0,7]}_{3\pm 1}^{\ge 1\uw}=\tra{(0,7)}{\llbracket 0,7]}_{3\pm 1}\medskip\\

\begin{tikzpicture}[scale=0.7]
\uwcrossing{1}
\dwcrossingc{3}{violet}
\uwcrossingfillc{3}{orange}
\drawarray{{0,2,3,6,7}}{{0,0,3,5,7}}
\ineqeq{4}{4}
\end{tikzpicture}
&&
\begin{tikzpicture}[scale=0.7]
\uwcrossing{1}
\dwcrossingfillc{3}{violet}
\uwcrossingc{3}{orange}
\drawarray{{0,2,3,6,7}}{{0,0,3,5,7}}
\ineqxeq{4}{4}
\end{tikzpicture}
&&
\begin{tikzpicture}[scale=0.7]
\uwcrossingfill{1}
\dwcrossingc{3}{violet}
\drawarray{{0,2,3,7}}{{0,0,3,5,6,7}}
\ineqeq{3}{5}
\end{tikzpicture}
&&
\begin{tikzpicture}[scale=0.7]
\uwcrossing{1}
\drawarray{{0,2,3,5,6,7}}{{0,0,3,7}}
\ineqxeq{5}{3}
\end{tikzpicture}
\end{array}$
\caption{An example of the bijection~\eqref{eq:for_figIX}, where $(x,y)=(0,0)$, $(u,v)=(7,7)$, $m=1$ and $n=3$.}
\label{fig:CaseIX}
\end{figure}

Adding Equations~\eqref{eq:Neven} and~\eqref{eq:Eeven} to account for all paths with at least $2m$ crossings, we get
\begin{align*}\sum_{\cd\in\tra{(x,u]}{[y,v)}_{n}^{\ge 2m}} q^{\su\bfc+\su\bfd-n(x+y)}&=
q^{n^2+m(m+1)}\left(q^{n-m}\qbin{a}{n-m}\qbin{a-1}{n+m}+\qbin{a}{n+m}\qbin{a-1}{n-m-1}\right)\\
&=q^{n^2+m(m+1)}\qbin{a}{n+m}\qbin{a}{n-m}\frac{q^{n-m}(1-q^{a-n-m})+1-q^{n-m}}{1-q^a}\\
&=q^{n^2+m(m+1)}\frac{1-q^{a-2m}}{1-q^a}\qbin{a}{n+m}\qbin{a}{n-m},
\end{align*}
which proves Equation~\eqref{eq:0=l=a-b:des-even}.
 
Similarly, adding Equations~\eqref{eq:Nodd} and~\eqref{eq:Eodd} to account for all paths with at least $2m+1$ crossings, we get
\begin{align*}\sum_{\cd\in\tra{(x,u]}{[y,v)}_{n}^{\ge 2m+1}} &  q^{\su\bfc+\su\bfd-n(x+y)}\\
&=q^{n^2+m(m+1)}\left(q^{n+m+1}\qbin{a-1}{n-m-1}\qbin{a}{n+m+1}+\qbin{a-1}{n+m}\qbin{a}{n-m-1}\right)\\
&=q^{n^2+m(m+1)}\qbin{a}{n+m+1}\qbin{a}{n-m-1}\frac{q^{n+m+1}(1-q^{a-n+m+1})+1-q^{n+m+1}}{1-q^a}\\
&=q^{n^2+m(m+1)}\frac{1-q^{a+2(m+1)}}{1-q^a}\qbin{a}{n+m+1}\qbin{a}{n-m-1},
\end{align*}
which proves Equation~\eqref{eq:0=l=a-b:des-odd}.
\end{list1}

\section{Proofs for paths crossing each other} \label{sec:proofs-pairs}

In this section we prove Theorem~\ref{thm:pairs:des}. Using the bijection~\eqref{bij:Pto2ra}, we will encode pairs of lattice paths as pairs of two-rowed arrays, describe crossings in this setting, and then define certain bijections on pairs of arrays.

\subsection{Pairs of two-rowed arrays}

Throughout the section, let $x_1,y_1,x_2,y_2,u_1,v_1,u_2,v_2,k\in\Z$ and $n\ge0$, let $\{\ci,\bu\}=\{1,2\}$, and let
$A_1=(x_1,y_1)$, $A_2=(x_2,y_2)$, $B_1=(u_1,v_1)$ and $B_2=(u_2,v_2)$.
We consider certain sets of pairs of two-rowed arrays, for which we introduce the notation
\begin{equation}\label{def:ptra}
\ptra{(x_1,u_\ci]}{[y_1,v_\ci)}{(x_2,u_\bu]}{[y_2,v_\bu)}_{n,k}=\bigcup_{n_1+n_2=n}\tra{(x_1,u_\ci]_{n_1}}{[y_1,v_\ci)_{n_1+k}}\times\tra{(x_2,u_\bu]_{n_2}}{[y_2,v_\bu)_{n_2-k}}.
\end{equation}
Elements of such sets are denoted by placing two two-rowed arrays side by side, namely
$\cdef$, where $\cd\in\tra{(x_1,u_\ci]_{n_1}}{[y_1,v_\ci)_{n_1+k}}$ and $\ef\in\tra{(x_2,u_\bu]_{n_2}}{[y_2,v_\bu)_{n_2-k}}$, with $n_1+n_2=n$. When $k=0$, the subscript $k$ will be often omitted.

Applying the encoding~\eqref{bij:Pto2ra} to each component of a pair of paths, we get a bijection 
\begin{equation}\label{bij:PQto2ra}
\{(P,Q)\in\P_{A_1\to B_\ci}\times \P_{A_2\to B_\bu}:\des(P)+\des(Q)=n\}\to\ptra{(x_1,u_\ci]}{[y_1,v_\ci)}{(x_2,u_\bu]}{[y_2,v_\bu)}_{n}.
\end{equation}
See Figure~\ref{fig:pair-arrays} for an example.
Suppose that condition~\eqref{condition} holds, and let $z=x_1+y_1=x_2+y_2$. If $(P,Q)$ is encoded by $\cdef$, then
\begin{equation}\label{eq:majPQ2ra}\maj(P)+\maj(Q)=\sum_{i=1}^{n_1} (c_i+d_i-x_1-y_1)
+\sum_{j=1}^{n_2} (e_j+f_j-x_2-y_2)
=\su\bfc+\su\bfd+\su\bfe+\su\bff-nz.
\end{equation}
Next we adapt Lemma~\ref{lem:sum} to enumerate the sets~\eqref{def:ptra} with respect to this statistic.

\begin{figure}[htb]
\centering
\begin{tikzpicture}[scale=0.6]
\draw (-.3,0)--(10.3,0);
\draw (0,-.3)--(0,8.3);
\markx{0}{2}
\markx{3}{2}
\markx{6}{4}
\markx{10}{7}
\markx{2}{0}
\markx{4}{5}
\markx{7}{6}
\markx{8}{7}
\markx{8}{8}
\diamant{3}{2}
\draw[teal,below right] (3,2) node[scale=.8] {$(3,2)$};
\diamant{6}{4}
\draw[teal,below right] (6,4) node[scale=.8] {$(6,4)$};
\diamant{4}{5}
\draw[teal,below right] (4,5) node[scale=.8] {$(4,5)$};
\diamant{7}{6}
\draw[teal,below right] (7,6) node[scale=.8] {$(7,6)$};
\diamant{8}{7}
\draw[teal,below right] (8,7) node[scale=.8] {$(8,7)$};
    \draw[red,very thick,fill](0,2) \start\E\E\E\N\N\E\E\E\N\N\N\E\E\E\E;
    \draw[blue,dashed,very thick,fill](2,0) \start\N\N\E\N\N\N\E\N\E\E\E\N\E\N;
\crossingc{3}{4}{olive};\crossingc{6}{6}{violet};\crossingc{8}{7}{orange};	
\draw[olive,above left] (3,4) node[scale=.8] {$(3,4)$};
\draw[violet,above left] (6,6) node[scale=.8] {$(6,6)$};
\draw[orange,above left] (8,7) node[scale=.8] {$(8,7)$};
	\draw (0,2) node[below right=-.5mm] {$A_1$};
	\draw (2,0) node[above left=-.5mm] {$A_2$};     
	\draw (8,8) node[right] {$B_1$};
	\draw (10,7) node[right] {$B_2$};     
	    \draw[red] (1,2.5) node {$P$} ;
    \draw[blue] (2.5,1) node {$Q$} ;

\draw[<->] (10.75,4)--(11.5,4);
\begin{scope}[shift={(12.5,3.5)},scale=1.17]
\drawarray{{0,3,6,10}}{{2,2,4,7}}
\ineq{3}{3}
\dwcrossingc{2}{olive}
\uwcrossingc{2}{violet}
\dwcrossingc{3 }{orange}
\draw (4,-.2)--(4,1);
\begin{scope}[shift={(4.5,0)}]
\drawarray{{2,3,4,7,8,8}}{{0,2,5,6,7,8}}
\ineq{5}{5}
\uwcrossingc{1}{olive}
\dwcrossingc{3}{violet}
\uwcrossingc{4}{orange}
\end{scope}
\draw (10.5,-1.2) node[left] {$\in\ptra{(0,10]}{[2,7)}{(2,8]}{[0,8)}_{6}$};
\end{scope}

\end{tikzpicture}
\caption{The encoding~\eqref{bij:PQto2ra} applied to the pair of paths from Figure~\ref{fig:pair}, and the resulting pair of two-rowed arrays, where the crossings have been circled.
}
\label{fig:pair-arrays}
\end{figure}
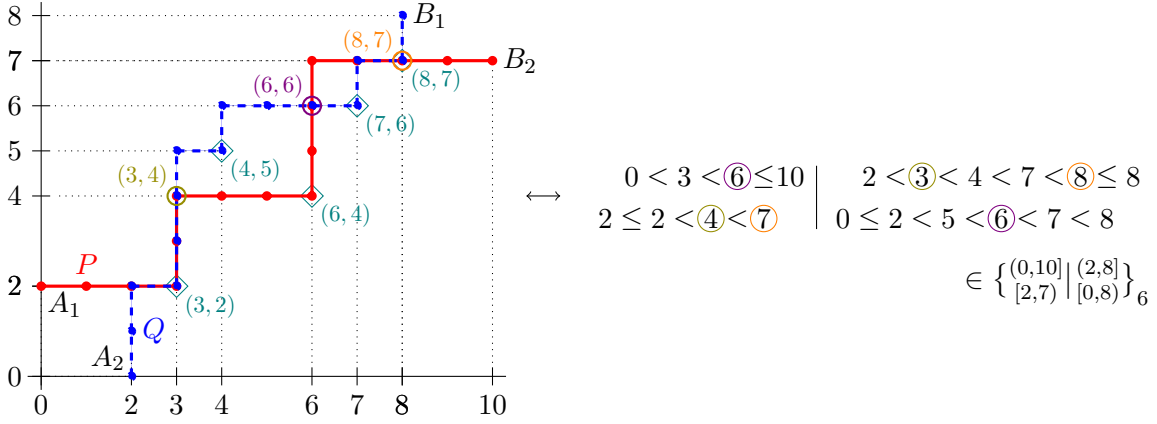

\begin{lemma}\label{lem:sum-pairs}
Suppose that $z=x_1+y_1=x_2+y_2$. We have
\begin{align*}
&\sum_{n\ge0}t^n\sum_{\cdef\in\ptra{(x_1,u_\ci]}{[y_1,v_\ci)}{(x_2,u_\bu]}{[y_2,v_\bu)}_{n,k}} q^{\su\bfc+\su\bfd+\su\bfe+\su\bff-nz}\\
&=q^{k(k+x_2-x_1)} \left(\sum_{n_1\ge0}t^{n_1} q^{n_1(n_1+k)}\qbin{u_\ci-x_1}{n_1}\qbin{v_\ci-y_1}{n_1+k}\right)
\left(\sum_{n_2\ge0}t^{n_2} q^{n_2(n_2-k)}\qbin{u_\bu-x_2}{n_2}\qbin{v_\bu-y_2}{n_2-k}\right)\\
&=f_{k,A_1,A_2,B_\ci,B_\bu}(t,q).
\end{align*}
\end{lemma}

\begin{proof}
Using~\eqref{def:ptra}, the left-hand side expression can be factored as
\begin{equation}\label{eq:n1n2}\left(\sum_{n_1\ge0}t^{n_1}\sum_{\cd\in\tra{(x_1,u_\ci]_{n_1}}{[y_1,v_\ci)_{n_1+k}}}q^{\su\bfc+\su\bfd-n_1z}\right)
\left(\sum_{n_2\ge0}t^{n_2}\sum_{\cd\in\tra{(x_2,u_\bu]_{n_2}}{[y_2,v_\bu)_{n_2-k}}} q^{\su\bfc+\su\bfd-n_2z}\right)
\end{equation}
For fixed $n_1$, Lemma~\ref{lem:sum}(i) gives
\begin{align}\nonumber
\sum_{\cd\in\tra{(x_1,u_\ci]_{n_1}}{[y_1,v_\ci)_{n_1+k}}}q^{\su\bfc+\su\bfd-n_1z}
&=\left(\sum_{\bfc\in(x_1,u_\ci]_{n_1}}q^{\su\bfc}\right)
\left(\sum_{\bfd\in[y_1,v_\ci)_{n_1+k}}q^{\su\bfd}\right)q^{-n_1(x_1+y_1)}\\
&=q^{n_1(n_1+k)+\binom{k}{2}+ky_1}\qbin{u_\ci-x_1}{n_1}\qbin{v_\ci-y_1}{n_1+k},
\label{eq:n1}
\end{align}
where we used the simplification
$$\binom{n_1+1}{2}+\binom{n_1+k+1}{2}+n_1x_1+(n_1+k)(y_1-1)-n_1(x_1+y_1)=n_1(n_1+k)+\binom{k}{2}+ky_1.$$
Similarly,
\begin{equation}\label{eq:n2}
\sum_{\cd\in\tra{(x_2,u_\bu]_{n_2}}{[y_2,v_\bu)_{n_2-k}}} q^{\su\bfc+\su\bfd-n_2z}
=q^{n_2(n_2-k)+\binom{k+1}{2}-ky_2}\qbin{u_\bu-x_2}{n_2}\qbin{v_\bu-y_2}{n_2-k}.
\end{equation}

Substituting~\eqref{eq:n1} and~\eqref{eq:n2} into~\eqref{eq:n1n2} and using that $\binom{k}{2}+\binom{k+1}{2}+k(y_1-y_2)=k(k+x_2-x_1)$, we obtained the stated identity.
\end{proof}

\subsection{Crossings in pairs of two-rowed arrays}

Let vertex $V_s$ be a crossing of two paths $P$ and $Q$, as defined in Section~\ref{sec:pairs-crossing}. We say that $V_s$ is an {\em upward} (resp.\ {\em downward}) crossing of $(P,Q)$ if the step of $P$ leaving $V_s$ is an $N$ (resp.\ $E$); equivalently, if the step of $Q$ leaving $V_s$ is an $E$ (resp.\ $N$).

Crossings of a pair of paths can be read from their encoding~\eqref{bij:PQto2ra} as a pair of two-rowed arrays. Indeed, suppose that $(P,Q)$ is encoded by $\cdef$, 
where $\cd\in\tra{(x_1,u_\ci]_{n_1}}{[y_1,v_\ci)_{n_1}}$ and $\ef\in\tra{(x_2,u_\bu]_{n_2}}{[y_2,v_\bu)_{n_2}}$, and let $c_0:=x_1$, $d_0:=y_1$, $c_{n_1+1}:=u_\ci$, $d_{n_1+1}:=v_\ci$, $e_0:=x_2$, $f_0:=y_2$, $e_{n_2+1}:=u_\bu$, $f_{n_2+1}:=v_\bu$ by convention. For simplicity, let us assume that $A_1\prec A_2$ or $A_1=A_2$. Then $(P,Q)$ has an
upward crossing at $(c_i,f_j)$, where $0\le i\le n_1$ and $1\le j\le n_2+1$, if all of the following hold:
\begin{enumerate}[(i$^\uw$)]
\item $e_{j-1}\le c_i<e_j$ and $d_i\le f_j<d_{i+1}$,
\item $(e_{j-1},f_{j-1},e_{j-2},f_{j-2},\dots,e_0,f_0)\alt(c_i,d_i,c_{i-1},d_{i-1},\dots,c_0,d_0)$ and\\
$(d_i,c_{i-1},d_{i-1},c_{i-2},\dots,c_0,d_0)\alt(f_j,e_{j-1},f_{j-1},e_{j-2},\dots,e_0,f_0)$,
\end{enumerate}
where $\alt$ is defined recursively by $(a_1,a_2,a_3,\dots)\alt(b_1,b_2,b_3,\dots)$ if either $a_1<b_1$, or $a_1=b_1$ and $(b_2,b_3,\dots)\alt(a_2,a_3,\dots)$.
Indeed, condition (i$^\uw$) states that $(c_i,f_j)$ belongs to both $P$ and $Q$, and that $P$ (resp.\ $Q$) leaves this vertex with with an $N$ (resp.\ $E$).
Condition (ii$^\uw$) states that, if $V_1$ is the first vertex of the maximal sequence of consecutive common vertices ending at $(c_i,f_j)$, then $P$ (resp.\ $Q$) arrives at $V_1$ with an $N$ (resp.\ $E$).

Similarly, $(P,Q)$ has a downward crossing at $(e_j,d_i)$, where $1\le i\le n_1+1$ and $0\le j\le n_2$, if
\begin{enumerate}[(i$^\dw$)]
\item $c_{i-1}\le e_j<c_{i}$ and $f_{j}\le d_i<f_{j+1}$,
\item $(c_{i-1},d_{i-1},c_{i-2},d_{i-2},\dots,c_0,d_0)\alt(e_j,f_j,e_{j-1},f_{j-1},\dots,e_0,f_0)$ and\\ 
$(f_j,e_{j-1},f_{j-1},e_{j-2},\dots,c_0,d_0)\alt(d_i,c_{i-1},d_{i-1},c_{i-2},\dots,e_0,f_0)$.
\end{enumerate}
For example, the pair of paths in Figure~\ref{fig:pair-arrays} has a downward crossing at $(e_1,d_2)=(3,4)$. Condition (i$^\dw$) states that $3\le 3<6$ and $2\le 4<5$, and condition (ii$^\dw$) states that
$(3,2,0,2)\alt(3,2,2,0)$ and $(2,2,0)\alt(4,3,2,0,2)$.

Next we generalize the definition of upward and downward crossings to pairs of two-rowed arrays $\cdef\in\ptra{(x_1,u_\ci]}{[y_1,v_\ci)}{(x_2,u_\bu]}{[y_2,v_\bu)}_{n,k}$ with $k\in\Z$. 
Suppose that $\cd\in\tra{(x_1,u_\ci]_{n_1}}{[y_1,v_\ci)_{n_1+k}}$ and $\ef\in\tra{(x_2,u_\bu]_{n_2}}{[y_2,v_\bu)_{n_2-k}}$, where $n_1+n_2=n$, and use the convention $c_0:=x_1$, $d_0:=y_1$, $c_{n_1+1}:=u_\ci$, $d_{n_1+k+1}:=v_\ci$, $e_0:=x_2$, $f_0:=y_2$, $e_{n_2+1}:=u_\bu$, $f_{n_2-k+1}:=v_\bu$. Let $m_1=\min(n_1,n_1+k)$ and $m_2=\min(n_2,n_2-k)$.
Then $\cdef$ has an upward crossing at $(c_i,f_j)$ if $0\le i\le m_1$ and $1\le j\le m_2+1$, and conditions (i$^\uw$) and (ii$^\uw$) hold. Similarly, it has a downward crossing at $(e_j,d_i)$ if $1\le i\le m_1+1$ and $0\le j\le m_2$, and conditions (i$^\dw$) and (ii$^\dw$) hold.

It is convenient to think of crossings of a pair of two-rowed arrays as crossings of the pair of paths obtained by truncating the arrays, similarly to what we did in Section~\ref{sec:crossings} for single arrays. 
Let $\Tr(\cd)$ be the path in $\P_{(x_1,y_1)\to(c_{m_1+1},d_{m_1+1})}$ having valleys at positions $(c_i,d_i)$ for $1\le i\le m_1$, and let $\Tr(\ef)$ be the path in $\P_{(x_2,y_2)\to(e_{m_2+1},f_{m_2+1})}$ having valleys at positions $(e_j,f_j)$ for $1\le j\le m_2$. 
Then the upward and downward crossings of $\cdef$ can be identified with the upward and downward crossings of the pair of paths $(\Tr(\cd),\Tr(\ef))$. See Figure~\ref{fig:truncated-pairs} for an example.
In particular, upward crossings are always at vertices of the form $(c_i,f_j)$, and downward crossings are at vertices of the form $(e_j,d_i)$, for some $i,j$. 

\begin{figure}[htb]
\centering
\begin{tikzpicture}[scale=0.6]
\begin{scope}[shift={(-2,5.5)},scale=1.17]
\drawarray{{0,3}}{{1,2,5}}
\draw (1,.8) node {$\le$};
\draw (.5,0) node {$\le$};
\draw (1.5,0) node {$<$};
\dwcrossingc{1}{olive}
\draw (2.5,-.2)--(2.5,1);
\begin{scope}[shift={(3,0)}]
\drawarray{{1,2,4,4}}{{0,2,4}}
\ineq{3}{2}
\uwcrossingc{1}{olive}
\end{scope}
\draw (6.7,.4) node[right] {$\in\ptra{(0,3]}{[1,5)}{(1,4]}{[0,4)}_{3,1}$};
\end{scope}

\draw (-.3,0)--(4.3,0);
\draw (0,-.3)--(0,4.3);
\diamant{2}{2}
    \draw[red,very thick,fill](0,1) \start\N\E\E\E;
    \draw[blue,dashed,very thick,fill](1,0) \start\N\N\E\N\N\E\E;
\crossingc{2}{2}{olive};
\draw[olive, above left] (2,2) node[scale=.8] {$(2,2)$};
\draw[left] (0,1) node[scale=.8] {$(0,1)$};
\draw[right] (3,2) node[scale=.8] {$(3,2)$};
\draw[below] (1,0) node[scale=.8] {$(1,0)$};
\draw[right] (4,4) node[scale=.8] {$(4,4)$};
\end{tikzpicture}
\caption{A pair of two-rowed arrays $\cdef$ with its crossing circled, and the corresponding pair of paths $(\Tr(\cd),\Tr(\ef))$.}
\label{fig:truncated-pairs}
\end{figure}

It is clear from this description that there is a natural ordering of the crossings by increasing $x$-coordinate, or equivalently, by increasing $y$-coordinate.
As in the case of single arrays, the {\em $r$th crossing} of a pair of two-rowed arrays will always refer to the $r$th crossing in this ordering.

For $r\ge0$, denote by $\ptra{(x_1,u_\ci]}{[y_1,v_\ci)}{(x_2,u_\bu]}{[y_2,v_\bu)}_{n,k}^{\ge r}$ the 
subset of $\ptra{(x_1,u_\ci]}{[y_1,v_\ci)}{(x_2,u_\bu]}{[y_2,v_\bu)}_{n,k}$ consisting of pairs of arrays that have at least $r$ crossings. 
The encoding~\eqref{bij:PQto2ra} restricts to a bijection
\begin{equation}\label{bij:PQto2ra:cro}
\{\pathsP{A_1}{B_\ci}{A_2}{B_\bu}{r}:\des(P)+\des(Q)=n\}\to\ptra{(x_1,u_\ci]}{[y_1,v_\ci)}{(x_2,u_\bu]}{[y_2,v_\bu)}_{n}^{\ge r}.
\end{equation}
Using Equation~\eqref{eq:majPQ2ra}, it follows that, if $z=x_1+y_1=x_2+y_2$, then
\begin{equation}\label{eq:Hto2ra}
H^{\ge r}_{A_1\to B_\ci,A_2\to B_\bu}(t,q)=\sum_{n\ge0}\,t^n\sum_{\cdef\in\ptra{(x_1,u_\ci]}{[y_1,v_\ci)}{(x_2,u_\bu]}{[y_2,v_\bu)}_{n}^{\ge r}} q^{\su\bfc+\su\bfd+\su\bfe+\su\bff-nz}.
\end{equation}

To prove Theorem~\ref{thm:pairs:des}, we will construct bijections between $\ptra{(x_1,u_\ci]}{[y_1,v_\ci)}{(x_2,u_\bu]}{[y_2,v_\bu)}_{n}^{\ge r}$ and sets of the form 
$\ptra{(x_1,u_2]}{[y_1,v_2)}{(x_2,u_1]}{[y_2,v_1)}_{n,k}$ for some $k\in\Z$, and then apply Lemma~\ref{lem:sum-pairs}.

\begin{lemma}\label{lem:first_crossing-pairs}
Let $r\ge1$. If $A_1\prec A_2$, then the $r$th crossing of a pair of arrays $\cdef\in\ptra{(x_1,u_\ci]}{[y_1,v_\ci)}{(x_2,u_\bu]}{[y_2,v_\bu)}_{n,k}^{\ge r}$ is a downward crossing if $r$ is odd, and an upward crossing if $r$ is even. 
\end{lemma}

\begin{proof} 
Interpreting crossings of $\cdef$ as crossings of the pair of paths $(\Tr(\cd),\Tr(\ef))$, which start at $A_1$ and $A_2$, respectively, the fact that $A_1\prec A_2$ implies that downward and upward crossings must alternate, with the first crossing being downward.
\end{proof}

For $r\ge1$, a symbol $\uw$ (resp.\ $\dw$) next to the superscript ${\ge}r$ denotes the subset of pairs of arrays
where the $r$th crossing is an upward (resp.\ downward) crossing.
For $r=0$, in the case $A_1\prec A_2$, we simplify define 
\begin{equation}\label{eq:convention_r0-pairs}
\ptra{(x_1,u_\ci]}{[y_1,v_\ci)}{(x_2,u_\bu]}{[y_2,v_\bu)}_{n,k}^{\ge 0\uw}=\ptra{(x_1,u_\ci]}{[y_1,v_\ci)}{(x_2,u_\bu]}{[y_2,v_\bu)}_{n,k}^{\ge 0\dw}=\ptra{(x_1,u_\ci]}{[y_1,v_\ci)}{(x_2,u_\bu]}{[y_2,v_\bu)}_{n,k}^{\ge 0}=\ptra{(x_1,u_\ci]}{[y_1,v_\ci)}{(x_2,u_\bu]}{[y_2,v_\bu)}_{n,k}
\end{equation}
by convention.

In the case $A_1=A_2=(x,y)$ and $B_1=B_2=(u,v)$, we define $\ptra{(x,u]}{[y,v)}{(x,u]}{[y,v)}_{n,k}^{\ge 0\uw}$ (resp.\ $\ptra{(x,u]}{[y,v)}{(x,u]}{[y,v)}_{n,k}^{\ge 0\dw}$) to be the set of pairs $\cdef\in\ptra{(x,u]}{[y,v)}{(x,u]}{[y,v)}_{n,k}$ such that $\cd\neq\ef$ and 
the leftmost entry in the usual zig-zag order where $\cd$ and $\ef$ differ is in the top row and satisfies $c_i<e_i$ (resp.\ $c_i>e_i$), or it is in the bottom row and satisfies $d_i>f_i$ (resp.\ $d_i<f_i$). 
Equivalently, $\cdef$ is  in the first (resp.\ second) set if the first step where the paths $\Tr(\cd)$ and $\Tr(\ef)$ disagree is an $N$ (resp.\ $E$) step of $\Tr(\cd)$ and an $E$ (resp.\ $N$) step of $\Tr(\ef)$. See the examples in Figure~\ref{fig:0uw}.

In analogy to Lemma~\ref{lem:even=odd} for single arrays, the next lemma shows how the relative locations of the two initial points and of the two final points often force the number of crossings of a pair of two-rowed arrays to have a given parity.

\begin{lemma}\label{lem:even=odd-pairs}
Let $m\ge0$. 
\begin{enumerate}[(a)]
\item If $A_1\prec A_2$, $B_1\prec B_2$, and $s\ge0$, then
\begin{align}\label{eq:odd=even}
&\ptra{(x_1,u_2]}{[y_1,v_2)}{(x_2,u_1]}{[y_2,v_1)}_{n,s}^{\ge 2m+1}=\ptra{(x_1,u_2]}{[y_1,v_2)}{(x_2,u_1]}{[y_2,v_1)}_{n,s}^{\ge 2m+1\dw}=\ptra{(x_1,u_2]}{[y_1,v_2)}{(x_2,u_1]}{[y_2,v_1)}_{n,s}^{\ge 2m\uw}=\ptra{(x_1,u_2]}{[y_1,v_2)}{(x_2,u_1]}{[y_2,v_1)}_{n,s}^{\ge 2m},&\\
\label{eq:even=odd}
&\ptra{(x_1,u_1]}{[y_1,v_1)}{(x_2,u_2]}{[y_2,v_2)}_{n,-s}^{\ge 2m+2}=\ptra{(x_1,u_1]}{[y_1,v_1)}{(x_2,u_2]}{[y_2,v_2)}_{n,-s}^{\ge 2m+2\uw}
=\ptra{(x_1,u_1]}{[y_1,v_1)}{(x_2,u_2]}{[y_2,v_2)}_{n,-s}^{\ge 2m+1\dw}=\ptra{(x_1,u_1]}{[y_1,v_1)}{(x_2,u_2]}{[y_2,v_2)}_{n,-s}^{\ge 2m+1}.&
\end{align}
\item If $A_1\prec A_2$ and $B_1=B_2$, then~\eqref{eq:odd=even} and~\eqref{eq:even=odd} hold for $s\ge1$. 
\item If $A_1=A_2=(x,y)$, $B_1=B_2=(u,v)$, and $s\ge1$, then
\begin{equation}\label{eq:odd=evenAB}
\ptra{(x,u]}{[y,v)}{(x,u]}{[y,v)}_{n,s}^{\ge 2m+1\dw}=\ptra{(x,u]}{[y,v)}{(x,u]}{[y,v)}_{n,s}^{\ge 2m\uw}\quad\text{and}\quad
\ptra{(x,u]}{[y,v)}{(x,u]}{[y,v)}_{n,-s}^{\ge 2m+2\uw}
=\ptra{(x,u]}{[y,v)}{(x,u]}{[y,v)}_{n,-s}^{\ge 2m+1\dw}.
\end{equation}
\end{enumerate}
\end{lemma}

\begin{proof}
In each of Equations~\eqref{eq:odd=even} and~\eqref{eq:even=odd} for $A_1\prec A_2$, the two outer equalities  follow from Lemma~\ref{lem:first_crossing-pairs} (and convention~\eqref{eq:convention_r0-pairs} in the case $m=0$), and the left-hand side is trivially contained in the right-hand side. To prove the reverse containment, we will show that the parity of the number of crossings is forced in each case.

Let us first prove Equation~\eqref{eq:odd=even} with the hypotheses of part~(a). 
Let $\cdef\in\ptra{(x_1,u_2]}{[y_1,v_2)}{(x_2,u_1]}{[y_2,v_1)}_{n,s}$, so that 
$\cd\in\tra{(x_1,u_2]_{n_1}}{[y_1,v_2)_{n_1+s}}$ and $\ef\in\tra{(x_2,u_1]_{n_2}}{[y_2,v_1)_{n_2-s}}$ for some $n_1,n_2$ summing to $n$.
Crossings of $\cdef$ are crossings of the pair of paths $(\Tr(\cd),\Tr(\ef))\in\P_{A_1\to B_2'}\times\P_{A_2\to B_1'}$, where $B_2'=(u_2,d_{n_1+1})$ and $B_1'=(e_{n_2-s+1},v_1)$. Since $d_{n_1+1}\le v_2$ and $e_{n_2-s+1}\le u_1$, the condition $B_1\prec B_2$ implies that $B'_1\prec B'_2$. Thus, since $A_1\prec A_2$, the number of crossings of $\cdef$ must be odd, proving Equation~\eqref{eq:odd=even} in this case.

With the hypotheses of part~(b), letting $B_1=B_2=(u,v)$ and $s\ge1$, the same argument yields endpoints $B_2'=(u,d_{n_1+1})$ and $B_1'=(e_{n_2-s+1},v)$ with $d_{n_1+1}< d_{n_1+s+1}=v$ and $e_{n_2-s+1}\le u$. 
Let $\widetilde\Tr(\ef)$ be the path obtained by removing the run of $E$ steps at the end of $\Tr(\ef)$, which does not affect any crossings since $B_2'$ is strictly below these steps. Crossings of $\cdef$ are now crossings of $(\Tr(\cd),\widetilde\Tr(\ef))\in\P_{A_1\to B_2'}\times\P_{A_2\to B_1''}$, where $B_1''=(e_{n_2-s},v)$. Since $e_{n_2-s}<e_{n_2-s+1}\le u$, we have $B_1''\prec B_2'$, implying again that the number of crossings of $\cdef$ is odd, 
which proves Equation~\eqref{eq:odd=even} also in this case.

With the hypotheses of part~(c), the same argument gives a pair $(\Tr(\cd),\widetilde\Tr(\ef))\in\P_{A\to B_2'}\times\P_{A\to B_1''}$, where $B_1''\prec B_2'$ as before. Suppose that $\cdef\in\ptra{(x,u]}{[y,v)}{(x,u]}{[y,v)}_{n,s}^{\ge 2m\uw}$. For $m\ge1$, this means that the $2m$th crossing of $(\Tr(\cd),\widetilde\Tr(\ef))$ is an upward crossing; for $m=0$, this means that the first step where these paths disagree is an $N$ step of $\Tr(\cd)$ and an $E$ step of $\widetilde\Tr(\ef)$. The fact that $B_1''\prec B_2'$ forces these paths to cross, with the $2m+1$st crossing being a downward crossing, which proves the first equality in~\eqref{eq:odd=evenAB}.

The proof of Equation~\eqref{eq:even=odd} is similar. 
Let $\cdef\in\ptra{(x_1,u_1]}{[y_1,v_1)}{(x_2,u_2]}{[y_2,v_2)}_{n,-s}$, so that 
$\cd\in\tra{(x_1,u_1]_{n_1}}{[y_1,v_1)_{n_1-s}}$ and $\ef\in\tra{(x_2,u_2]_{n_2}}{[y_2,v_2)_{n_2+s}}$ for some $n_1,n_2$.
Crossings of $\cdef$ are crossings of $(\Tr(\cd),\Tr(\ef))\in\P_{A_1\to B_1'}\times\P_{A_2\to B_2'}$, where $B_1'=(c_{n_1-s+1},v_1)$ and $B_2'=(u_2,f_{n_2+1})$. Since $c_{n_1-s+1}\le u_1$ and $f_{n_2+1}\le v_2$, the condition $B_1\prec B_2$ implies that $B'_1\prec B'_2$.  Thus, if the hypothesis of part (a) hold, the number of crossings of $\cdef$ must be even, proving Equation~\eqref{eq:even=odd} in this case.

Letting now $B_1=B_2=(u,v)$ and $s\ge1$, we get endpoints $B_1'=(c_{n_1-s+1},v)$ and $B_2'=(u,f_{n_2+1})$ with 
$c_{n_1-s+1}\le u$ and $f_{n_2+1}<v$. Removing the run of $E$ steps at the end of $\Tr(\cd)$, which does not affect any crossings, we obtain a pair $(\widetilde\Tr(\cd),\Tr(\ef))\in\P_{A_1\to B_1''}\times\P_{A_2\to B_2'}$, where $B_1''=(c_{n_1-s},v)$. Now $c_{n_1-s}<c_{n_1-s+1}\le u$, and so $B_1''\prec B_2'$, implying again that the number of crossings of $\cdef$ is even. This proves Equation~\eqref{eq:even=odd} with the hypotheses of part~(b).

Finally, with the hypotheses of part~(c), we obtain a pair $(\widetilde\Tr(\cd),\Tr(\ef))\in\P_{A_1\to B_1''}\times\P_{A_2\to B_2'}$, where $B_1''\prec B_2'$. If $\cdef\in\ptra{(x,u]}{[y,v)}{(x,u]}{[y,v)}_{n,-s}^{\ge 2m+1\dw}$, the $2m+1$st crossing of $(\widetilde\Tr(\cd),\Tr(\ef))$ is a downward crossing, so the paths must cross again, and the $2m+2$nd crossing must be an upward crossing. This proves the second equality in~\eqref{eq:odd=evenAB}.
\end{proof}

\subsection{The bijections $\ga_r$ and $\de_r$}\label{sec:ga-de}

The bijections $\ga_r$ and $\de_r$ play a similar role for pairs of two-rowed arrays as the bijections $\al_r$ and $\be_r$ played for single arrays. They are reminiscent of the bijection $\theta_r$ defined in~\cite{part1} for pairs of paths; however, $\ga_r$ and $\de_r$ do not restrict to bijections for pairs of paths, since they change the relative lengths of the rows of the arrays.

For $\cdef\in\ptra{(x_1,u_\ci]}{[y_1,v_\ci)}{(x_2,u_\bu]}{[y_2,v_\bu)}_{n}$, we say that an upward crossing at $(c_i,f_j)$ (resp.\ a downward crossing at $(e_j,d_i)$) is {\em proper} if $c_i\neq u_\ci$ and $f_j\neq v_\bu$ (resp.\ $e_j\neq u_\bu$ and $d_i\neq v_\ci$), that is, neither entry equals the upper bound for its row.

For $r\ge1$, the map $\ga_r$ applies to pairs of two-rowed arrays $\cdef$ whose $r$th crossing is a proper upward crossing, say at $(c_i,f_j)$, and it swaps the entries to the right of $c_i$ in each row of the first array with the entries to the right of $f_j$ in each row of the second array. Schematically, we have:
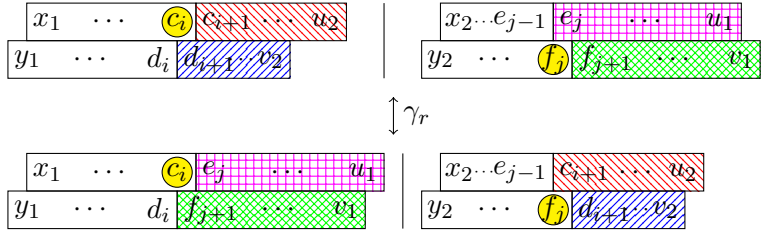
\begin{center}
\begin{tikzpicture}[scale=0.5]
\draw (0.5,1) rectangle (5,2);
\draw (.5,1.5) node[right=-1.5] {$x_1$};
\draw (2.75,1.5) node {$\cdots$};
\draw (0,0) rectangle (4.5,1);
\draw (0,.5) node[right=-1.5] {$y_1$};
\draw (2.25,0.5) node {$\cdots$};
\draw[fill=yellow] (4.5,1.5) circle (.4);
\draw (4.5,1.5) node {$c_i$};
\draw[pattern color=red,pattern=north west lines] (5,1) rectangle (9,2);
\draw[pattern color=blue,pattern=north east lines] (4.5,0) rectangle (7.5,1);
\draw (4.5,.5) node[left=-1.5] {$d_{i}$};
\draw (5,1.5) node[right=-1.5] {$c_{i+1}$};
\draw (7.25,1.5) node {$\cdots$};
\draw (4.5,.5) node[right=-1.5] {$d_{i+1}$};
\draw (6.25,0.5) node[scale=.8] {$\cdots$};
\draw (9,1.5) node[left=-1.5] {$u_2$};
\draw (7.5,.5) node[left=-1.5] {$v_2$};
\draw (10,0)--(10,2);

\begin{scope}[shift={(11,0)}]
\draw (0.5,1) rectangle (3.5,2);
\draw (1,1.5) node {$x_2$};
\draw (1.7,1.5) node[scale=.6] {$\cdots$};
\draw (0,0) rectangle (4,1);
\draw (.5,.5) node {$y_2$};
\draw (2,0.5) node {$\cdots$};
\draw[fill=yellow] (3.5,.5) circle (.4);
\draw (3.5,.5) node {$f_j$};
\draw[pattern color=magenta,pattern=grid] (3.5,1) rectangle (8.5,2);
\draw[pattern color=green,pattern=crosshatch] (4,0) rectangle (9,1);
\draw (3.5,1.5) node[left=-1.5] {$e_{j-1}$};
\draw (3.5,1.5) node[right=-1.5] {$e_{j}$};
\draw (6,1.5) node {$\cdots$};
\draw (4,.5) node[right=-1.5] {$f_{j+1}$};
\draw (6.75,.5) node {$\cdots$};
\draw (8.5,1.5) node[left=-1.5] {$u_1$};
\draw (9,.5) node[left=-1.5] {$v_1$};
\end{scope}

\draw[<->] (10.25,-.5) -- node[right]{$\ga_r$} (10.25,-1.5);

\begin{scope}[shift={(0,-4)}]
\draw (0.5,1) rectangle (5,2);
\draw (.5,1.5) node[right=-1.5] {$x_1$};
\draw (2.75,1.5) node {$\cdots$};
\draw (0,0) rectangle (4.5,1);
\draw (0,.5) node[right=-1.5] {$y_1$};
\draw (2.25,0.5) node {$\cdots$};
\draw[fill=yellow] (4.5,1.5) circle (.4);
\draw (4.5,1.5) node {$c_i$};
\draw[pattern color=magenta,pattern=grid] (5,1) rectangle (10,2);
\draw[pattern color=green,pattern=crosshatch] (4.5,0) rectangle (9.5,1);
\draw (4.5,.5) node[left=-1.5] {$d_{i}$};
\draw (5,1.5) node[right=-1.5] {$e_{j}$};
\draw (7.5,1.5) node {$\cdots$};
\draw (4.5,.5) node[right=-1.5] {$f_{j+1}$};
\draw (7.25,.5) node {$\cdots$};
\draw (10,1.5) node[left=-1.5] {$u_1$};
\draw (9.5,.5) node[left=-1.5] {$v_1$};
\draw (10.5,0)--(10.5,2);

\begin{scope}[shift={(11,0)}]
\draw (0.5,1) rectangle (3.5,2);
\draw (1,1.5) node {$x_2$};
\draw (1.7,1.5) node[scale=.6] {$\cdots$};
\draw (0,0) rectangle (4,1);
\draw (.5,.5) node {$y_2$};
\draw (2,0.5) node {$\cdots$};
\draw[fill=yellow] (3.5,.5) circle (.4);
\draw (3.5,.5) node {$f_j$};
\draw[pattern color=red,pattern=north west lines] (3.5,1) rectangle (7.5,2);
\draw[pattern color=blue,pattern=north east lines] (4,0) rectangle (7,1);
\draw (3.5,1.5) node[left=-1.5] {$e_{j-1}$};
\draw (3.5,1.5) node[right=-1.5] {$c_{i+1}$};
\draw (5.75,1.5) node {$\cdots$};
\draw (4,.5) node[right=-1.5] {$d_{i+1}$};
\draw (5.75,.5) node[scale=.8] {$\cdots$};
\draw (7.5,1.5) node[left=-1.5] {$u_2$};
\draw (7,.5) node[left=-1.5] {$v_2$};
\end{scope}
\end{scope}
\end{tikzpicture}
\end{center}
The fact that $(c_i,f_j)$ is a proper crossing of $\cdef$ guarantees that $c_{i+1}$ and $f_{j+1}$ exist, and that $c_i<c_{i+1}$ and $f_j<f_{j+1}$.
Condition (i$^\uw$) in the characterization of upward crossings implies that $c_i<e_j$, $d_i<f_{j+1}$, $e_{j-1}<c_{i+1}$ and $f_j<d_{i+1}$, and so the rows of the arrays in $\ga_r(\cdef)$ are increasing.
The pair $\ga_r(\cdef)$ still has a crossing at $(c_i,f_j)$: condition (i$^\uw$) holds because $e_{j-1}\le c_i<c_{i+1}$ and 
$d_i\le f_j<f_{j+1}$, and condition (ii$^\uw$) holds because the relevant entries are not affected by $\ga_r$.
This crossing is clearly proper, and it is the $r$th crossing of $\ga_r(\cdef)$ because the entries to the left of $c_i$ and $f_j$, and thus the first $r-1$ crossings of the pair of arrays, are not affected by $\ga_r$.
It follows that $\ga_r$ is an involution.

Similarly, the map $\de_r$ applies to pairs of two-rowed arrays $\cdef$ whose $r$th crossing is a proper downward crossing, say at $(e_j,d_i)$, and it again swaps the entries to the right of $d_i$ in the first array with the entries to the right of $e_j$ in the second array. Schematically, we have:
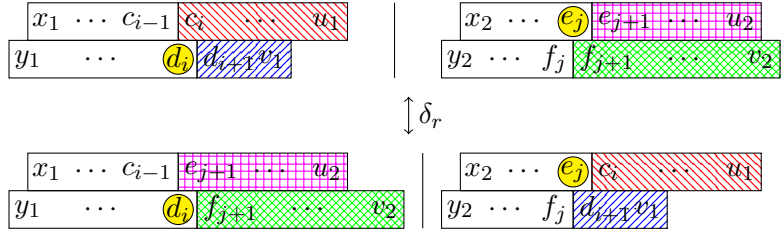
\begin{center}
\begin{tikzpicture}[scale=0.5]
\draw (0.5,1) rectangle (4.5,2);
\draw (1,1.5) node {$x_1$};
\draw (2.25,1.5) node {$\cdots$};
\draw (0,0) rectangle (5,1);
\draw (.5,.5) node {$y_1$};
\draw (2.5,.5) node {$\cdots$};
\draw[fill=yellow] (4.5,.5) circle (.4);
\draw (4.5,.5) node {$d_i$};
\draw[pattern color=red,pattern=north west lines] (4.5,1) rectangle (9,2);
\draw[pattern color=blue,pattern=north east lines] (5,0) rectangle (7.5,1);
\draw (4.5,1.5) node[left=-1.5] {$c_{i-1}$};
\draw (4.5,1.5) node[right=-1.5] {$c_{i}$};
\draw (6.75,1.5) node {$\cdots$};
\draw (5,.5) node[right=-1.5] {$d_{i+1}$};
\draw (6.5,0.5) node[scale=.6] {$\cdots$};
\draw (9,1.5) node[left=-1.5] {$u_1$};
\draw (7.5,.5) node[left=-1.5] {$v_1$};
\draw (10.25,0)--(10.25,2);

\begin{scope}[shift={(11.5,0)}]
\draw (0.5,1) rectangle (4,2);
\draw (1,1.5) node {$x_2$};
\draw (2.25,1.5) node {$\cdots$};
\draw (0,0) rectangle (3.5,1);
\draw (.5,.5) node {$y_2$};
\draw (1.75,.5) node {$\cdots$};
\draw[fill=yellow] (3.5,1.5) circle (.4);
\draw (3.5,1.5) node {$e_j$};
\draw[pattern color=magenta,pattern=grid] (4,1) rectangle (8.5,2);
\draw[pattern color=green,pattern=crosshatch] (3.5,0) rectangle (9,1);
\draw (3.5,.5) node[left=-1.5] {$f_j$};
\draw (4,1.5) node[right=-1.5] {$e_{j+1}$};
\draw (6.5,1.5) node {$\cdots$};
\draw (3.5,.5) node[right=-1.5] {$f_{j+1}$};
\draw (6.5,.5) node {$\cdots$};
\draw (8.5,1.5) node[left=-1.5] {$u_2$};
\draw (9,.5) node[left=-1.5] {$v_2$};
\end{scope}

\draw[<->] (10.625,-.5) -- node[right]{$\de_r$} (10.625,-1.5);

\begin{scope}[shift={(0,-4)}]
\draw (0.5,1) rectangle (4.5,2);
\draw (1,1.5) node {$x_1$};
\draw (2.25,1.5) node {$\cdots$};
\draw (0,0) rectangle (5,1);
\draw (.5,.5) node {$y_1$};
\draw (2.5,.5) node {$\cdots$};
\draw[fill=yellow] (4.5,.5) circle (.4);
\draw (4.5,.5) node {$d_i$};
\draw[pattern color=magenta,pattern=grid] (4.5,1) rectangle (9,2);
\draw[pattern color=green,pattern=crosshatch] (5,0) rectangle (10.5,1);
\draw (4.5,1.5) node[left=-1.5] {$c_{i-1}$};
\draw (4.5,1.5) node[right=-1.5] {$e_{j+1}$};
\draw (7,1.5) node {$\cdots$};
\draw (5,.5) node[right=-1.5] {$f_{j+1}$};
\draw (8,.5) node {$\cdots$};
\draw (9,1.5) node[left=-1.5] {$u_2$};
\draw (10.5,.5) node[left=-1.5] {$v_2$};
\draw (11,0)--(11,2);

\begin{scope}[shift={(11.5,0)}]
\draw (0.5,1) rectangle (4,2);
\draw (1,1.5) node {$x_2$};
\draw (2.25,1.5) node {$\cdots$};
\draw (0,0) rectangle (3.5,1);
\draw (.5,.5) node {$y_2$};
\draw (1.75,.5) node {$\cdots$};
\draw[fill=yellow] (3.5,1.5) circle (.4);
\draw (3.5,1.5) node {$e_j$};
\draw[pattern color=red,pattern=north west lines] (4,1) rectangle (8.5,2);
\draw[pattern color=blue,pattern=north east lines] (3.5,0) rectangle (6,1); 
\draw (3.5,.5) node[left=-1.5] {$f_j$};
\draw (4,1.5) node[right=-1.5] {$c_{i}$};
\draw (6.25,1.5) node {$\cdots$};
\draw (3.5,.5) node[right=-1.5] {$d_{i+1}$};
\draw (5,0.5) node[scale=.6] {$\cdots$};
\draw (8.5,1.5) node[left=-1.5] {$u_1$};
\draw (6,.5) node[left=-1.5] {$v_1$};
\end{scope}
\end{scope}
\end{tikzpicture}
\end{center}
The same argument shows that the rows of the arrays in $\de_r(\cdef)$ are increasing, that $\de_r(\cdef)$ still has a proper crossing at $(e_j,d_i)$, which is its $r$th crossing, and that the map $\de_r$ is an involution. In fact, if we denote by $\swap$ the involution that swaps the two two-rowed arrays in a pair, that is, 
\begin{equation}\label{eq:swap} \swap(\cdef)=\ptr{\bfe}{\bff}{\bfc}{\bfd}, \end{equation}
then the maps $\ga_r$ and $\de_r$ are related by $\de_r=\swap\circ\ga_r\circ\swap$. 

If $A_1=A_2=(x,y)$ and $B_1=B_2=(u,v)$, we can extend the definitions of $\ga_r$ and $\de_r$ to the case $r=0$ as follows. 
Let $\cdef\in\ptra{(x,u]}{[y,v)}{(x,u]}{[y,v)}_{n,k}^{\ge 0\uw}$.
If the leftmost entry where $\cd$ and $\ef$ differ is $c_i<e_i$, then the vertex $(c_i,f_i)$ satisfies condition (i$^\uw$) from the characterization of upward crossings,
since $e_{i-1}< c_i<e_i$ and $d_i=f_i<d_{i+1}$, even if it fails condition (ii$^\uw$) since $(d_i,c_{i-1},d_{i-1},\dots,c_0,d_0)=(f_i,e_{i-1},f_{i-1},\dots,e_0,f_0)$. 
If $c_i\neq u$ and $f_i\ne v$, we define $\ga_0$ by swapping the entries to the right of $c_i$ in each row of $\cd$ with the entries to the right of $f_i$ in each row of $\ef$, just as in the usual definition of $\ga_r$ if $(c_i,f_i)$ had been the $r$th crossing.

If the leftmost entry where $\cd$ and $\ef$ differ is $d_i>f_i$, now it is the vertex $(c_{i-1},f_i)$ that satisfies condition (i$^\uw$), since $e_{i-1}=c_{i-1}<e_i$ and $d_{i-1}< f_i<d_i$. If $c_{i-1}\neq u$ and $f_i\ne v$, we define $\ga_0$ by swapping the entries to the right of $c_{i-1}$ in each row of $\cd$ with the entries to the right of $f_i$ in each row of $\ef$, just as in the definition of $\ga_r$ if $(c_{i-1},f_i)$ had been the $r$th crossing. See the example in Figure~\ref{fig:0uw}.

The bijection $\de_0$ can be defined analogously, or as $\de_0=\swap\circ\ga_0\circ\swap$, but it will not be needed in the proofs.

\begin{figure}[htb]
\centering
\begin{tikzpicture}[scale=0.6]
\begin{scope}[shift={(-3,4.5)},scale=1.1] 
\drawarray{{0,2,4,4}}{{0,1,{\mathbf 3}}}
\ineq{3}{2}
\uwcrossingdot{1}
\draw (4,-.2)--(4,1);
\begin{scope}[shift={(4.5,0)}]
\drawarray{{0,2,4}}{{0,1,{\mathbf 2},3}}
\ineq{2}{3}
\dwcrossingdot{2}
\end{scope}
\draw (7.7,.4) node[right] {$\in\ptra{(0,4]}{[0,3)}{(0,4]}{[0,3)}_{3,-1}^{\ge 0\uw}$};
\draw[<->] (12.5,.4)-- node[above]{$\ga_0$} (13.2,.4);
\end{scope}

\draw (-.3,0)--(4.3,0);
\draw (0,-.3)--(0,3.3);
\diamant{2}{1}
    \draw[red,very thick,fill](0,0) \start\N\E\E\N\N\E\E;
    \draw[blue,dashed,very thick,fill](0,0) \start\N\E\E\N\E\E;
\draw[left] (0,0) node[scale=.8] {$(0,0)$};
\draw[right] (4,3) node[scale=.8] {$(4,3)$};
\draw[right] (4,2) node[scale=.8] {$(4,2)$};
\crossingdot{2}{2}

\begin{scope}[shift={(15,0)}]
\begin{scope}[shift={(-3,4.5)},scale=1.1]
\drawarray{{0,2,4}}{{0,1,{\mathbf 3}}}
\ineq{2}{2}
\uwcrossingdot {1}
\draw (3,-.2)--(3,1);
\begin{scope}[shift={(3.5,0)}]
\drawarray{{0,2,4,4}}{{0,1,{\mathbf 2},3}}
\ineq{3}{3}
\dwcrossingdot{2}
\end{scope}
\draw (7.2,.4) node[right] {$\in\ptra{(0,4]}{[0,3)}{(0,4]}{[0,3)}_{3,0}^{\ge 0\uw}$};
\end{scope}

\draw (-.3,0)--(4.3,0);
\draw (0,-.3)--(0,3.3);
\diamant{2}{1}
\diamant{4}{2}
    \draw[red,very thick,fill](0,0) \start\N\E\E\N\N\E\E;
    \draw[blue,dashed,very thick,fill](0,0) \start\N\E\E\N\E\E\N;
\draw[left] (0,0) node[scale=.8] {$(0,0)$};
\draw[right] (4,3) node[scale=.8] {$(4,3)$};
\crossingdot{2}{2}
\end{scope}

\end{tikzpicture}
\caption{An example of the bijection $\ga_0$. For each pair of two-rowed arrays $\cdef$, the leftmost entry where they differ is $d_2=3>2=f_2$, so $(c_1,f_2)=(2,2)$ satisfies condition (i$^\uw$). The corresponding vertex in the pair of paths $(\Tr(\cd),\Tr(\ef))$ has been marked with a dotted circle.}
\label{fig:0uw}
\end{figure}

\begin{lemma}\label{lem:ga,de}
Fix $n\ge0$, $k\in\Z$ and $r\ge1$.
Suppose that either $B_1\prec B_2$ and $k\ge0$, or that $B_1=B_2$.
The map $\ga_r$ restricts to a bijection
\begin{equation}\label{eq:ga}
\ptra{(x_1,u_2]}{[y_1,v_2)}{(x_2,u_1]}{[y_2,v_1)}_{n,k}^{\ge r\uw}\overset{\ga_r}{\longleftrightarrow}
\ptra{(x_1,u_1]}{[y_1,v_1)}{(x_2,u_2]}{[y_2,v_2)}_{n,-k-1}^{\ge r\uw}.
\end{equation}
The map $\de_r$ restricts to a bijection
\begin{equation}\label{eq:de}
\ptra{(x_1,u_1]}{[y_1,v_1)}{(x_2,u_2]}{[y_2,v_2)}_{n,-k}^{\ge r\dw}\overset{\de_r}{\longleftrightarrow}
\ptra{(x_1,u_2]}{[y_1,v_2)}{(x_2,u_1]}{[y_2,v_1)}_{n,k+1}^{\ge r\dw}.
\end{equation}
Both $\ga_r$ and $\de_r$ preserve the sum of the entries of the pair of arrays.

Additionally, if $A_1=A_2$ and $B_1=B_2$, then the above statements also hold for $r=0$.
\end{lemma}

\begin{proof} 
Suppose first that $r\ge1$. Let us first check that pairs of arrays in the four sets above cannot have improper crossings, and so the maps $\ga_r$ and $\de_r$ are defined.
For $\cdef\in\ptra{(x_1,u_\ci]}{[y_1,v_\ci)}{(x_2,u_\bu]}{[y_2,v_\bu)}_{n,h}$, where $h\in\Z$, to have an improper upward crossing at $(c_i,f_j)$, we must have either $c_i=u_\ci$, in which case $u_\ci=c_i<e_j\le u_\bu$ and $h\ge0$, or $f_j=v_\bu$, in which case $v_\bu=f_j<d_{i+1}\le v_\ci$ and $h\ge0$.
Similarly, for $\cdef$ to have an improper downward crossing at $(e_j,d_i)$, we must have either $e_j=u_\bu$, in which case $u_\bu=e_j<c_i\le u_\ci$ and $h\le0$, or $d_i=v_\ci$, in which case $v_\ci=d_i<f_{j+1}\le v_\bu$ and $h\le0$. 

If $\cdef$ is in the left-hand side of~\eqref{eq:ga} (resp.~\eqref{eq:de}) and has an improper upward (resp.\ downward) crossing, the previous paragraph forces $u_2<u_1$ or $v_2<v_1$, contradicting the hypothesis that $B_1\prec B_2$ or $B_1=B_2$.
If $\cdef$ is in the right-hand side instead,
the forced inequalities $u_1<u_2$ or $v_1<v_2$ hold when $B_1\prec B_2$, but the requirement on $h$ states that $k+1\le0$, contradicting the hypothesis that $k\ge0$ in this case.

Having already seen that $\ga_r$ and $\de_r$ are involutions preserving the first $r$ crossings and preserving the sum of the entries, it remains to describe their images when restricted to the above sets.
Let $\cdef$ be in one of the sets in~\eqref{eq:ga}, so that $\cd\in\tra{(x_1,u_\ci]_{n_1}}{[y_1,v_\ci)_{n_1+h}}$ and $\ef\in\tra{(x_2,u_\bu]_{n_2}}{[y_2,v_\bu)_{n_2-h}}$ for some $n_1,n_2$ summing to $n$, and $h\in\Z$. If the $r$th crossing of $\cdef$ is an upward crossing at $(c_i,f_j)$, then 
$$\ga_r(\cdef)\in\tra{(x_1,u_\bu]_{n_2+i-j+1}}{[y_1,v_\bu)_{n_2-h+i-j}}\times 
\tra{(x_2,u_\ci]_{n_1-i+j-1}}{[y_2,v_\ci)_{n_1+h-i+j}}\subseteq \ptra{(x_1,u_\bu]}{[y_1,v_\bu)}{(x_2,u_\ci]}{[y_2,v_\ci)}_{n,-h-1}.$$
When $h\in\{k,-k-1\}$, then $-h-1$ equals the other element in the set, so this argument works in both directions.

Similarly, if $\cdef$ is in one of the sets in~\eqref{eq:de} and its $r$th crossing is a downward crossing at $(e_j,d_i)$, then 
$$\de_r(\cdef)\in\tra{(x_1,u_\bu]_{n_2+i-j-1}}{[y_1,v_\bu)_{n_2-h+i-j}}\times 
\tra{(x_2,u_\ci]_{n_1-i+j+1}}{[y_2,v_\ci)_{n_1+h-i+j}}\subseteq\ptra{(x_1,u_\bu]}{[y_1,v_\bu)}{(x_2,u_\ci]}{[y_2,v_\ci)}_{n,-h+1}.$$
When $h\in\{-k,k+1\}$, then $-h+1$ equals the other element in the set.

Finally, in the case that $A_1=A_2$ and $B_1=B_2$, a similar argument shows that the maps $\ga_0$ and $\de_0$ are defined and they are bijections between the stated sets.
\end{proof}

\subsection{Proof of Theorem~\ref{thm:pairs:des}}

The proof is divided into four cases according to which endpoints of the paths coincide.
In each case, we determine 
$H^{\ge r}_{A_1\to B_\ci,A_2\to B_\bu}(t,q)$ by first using Equation~\eqref{eq:Hto2ra} to write it as a sum over pairs of two-rowed arrays.
Then we repeatedly apply the maps from Lemma~\ref{lem:ga,de} to construct bijections between $\ptra{(x_1,u_\ci]}{[y_1,v_\ci)}{(x_2,u_\bu]}{[y_2,v_\bu)}_{n}^{\ge r}$ and certain sets of pairs of two-rowed arrays with no requirement on the number of crossings, and finally we use Lemma~\ref{lem:sum-pairs} to obtain the desired expressions. Again, the cases are labeled as in~\cite{part1} for consistency.

\begin{list1}

\item {\bf Case 1:} endpoints $A_1\prec A_2$ and $B_1\prec B_2$.
If $P\in\P_{A_1\to B_2}$ and $Q\in\P_{A_2\to B_1}$, the relative position of the endpoints forces $\cro(P,Q)$ to be odd,
which proves the first equality in Equation~\eqref{eq:switched:des}.
Using Lemmas~\ref{lem:even=odd-pairs} and~\ref{lem:ga,de}, we construct a sequence of bijections $\de_1\circ\ga_2\circ\dots\circ\de_{2m-1}\circ\ga_{2m}$:
\begin{align}\nonumber\ptra{(x_1,u_2]}{[y_1,v_2)}{(x_2,u_1]}{[y_2,v_1)}_{n}^{\ge 2m}
=\ptra{(x_1,u_2]}{[y_1,v_2)}{(x_2,u_1]}{[y_2,v_1)}_{n,0}^{\ge 2m\uw}
\overset{\ga_{2m}}{\longrightarrow}
\ptra{(x_1,u_1]}{[y_1,v_1)}{(x_2,u_2]}{[y_2,v_2)}_{n,-1}^{\ge 2m\uw}
=\ptra{(x_1,u_1]}{[y_1,v_1)}{(x_2,u_2]}{[y_2,v_2)}_{n,-1}^{\ge 2m-1\dw}& \\ \nonumber
\overset{\de_{2m-1}}{\longrightarrow}
\ptra{(x_1,u_2]}{[y_1,v_2)}{(x_2,u_1]}{[y_2,v_1)}_{n,2}^{\ge 2m-1\dw}
=\ptra{(x_1,u_2]}{[y_1,v_2)}{(x_2,u_1]}{[y_2,v_1)}_{n,2}^{\ge 2m-2\uw}&\\ \label{eq:for_fig1}
\overset{\ga_{2m-2}}{\longrightarrow}\cdots\overset{\de_{1}}{\longrightarrow}
\ptra{(x_1,u_2]}{[y_1,v_2)}{(x_2,u_1]}{[y_2,v_1)}_{n,2m}^{\ge 1\dw}
=\ptra{(x_1,u_2]}{[y_1,v_2)}{(x_2,u_1]}{[y_2,v_1)}_{n,2m}^{\ge 0\uw}
=\ptra{(x_1,u_2]}{[y_1,v_2)}{(x_2,u_1]}{[y_2,v_1)}_{n,2m},&
\end{align}
where the last equality comes from~\eqref{eq:convention_r0-pairs}.
See Figure~\ref{fig:Case1} for an example.
Since these bijections preserve the sum of the entries of the arrays, Equation~\eqref{eq:Hto2ra} and Lemma~\ref{lem:sum-pairs} give
$$H^{\ge 2m}_{A_1\to B_2,A_2\to B_1}(t,q)=\sum_{n\ge0}t^n\sum_{\cdef\in\ptra{(x_1,u_2]}{[y_1,v_2)}{(x_2,u_1]}{[y_2,v_1)}_{n,2m}} q^{\su\bfc+\su\bfd+\su\bfe+\su\bff-nz}=f_{2m,A_1,A_2,B_2,B_1}(t,q),
$$
proving Equation~\eqref{eq:switched:des}.

\begin{figure}[htb]
\centering\medskip
$\begin{array}{ccc}
\ptra{(0,10]}{[2,7)}{(2,8]}{[0,8)}_{6}^{\ge2}=\ptra{(0,10]}{[2,7)}{(2,8]}{[0,8)}_{6,0}^{\ge2\uw}&
\overset{\ga_{2}}{\longrightarrow}&
\ptra{(0,8]}{[2,8)}{(2,10]}{[0,7)}_{6,-1}^{\ge2\uw}=\ptra{(0,8]}{[2,8)}{(2,10]}{[0,7)}_{6,-1}^{\ge1\dw}\medskip\\
\begin{tikzpicture}[scale=0.7]
\dwcrossingc{2}{olive}
\uwcrossingfillc{2}{violet}
\dwcrossingc{3 }{orange}
\drawarray{{0,3,6,10}}{{2,2,4,7}}
\ineq{3}{3}
\draw (4,-.2)--(4,1);
\begin{scope}[shift={(4.5,0)}]
\uwcrossingc{1}{olive}
\dwcrossingfillc{3}{violet}
\uwcrossingc{4}{orange}
\drawarray{{2,3,4,7,8,8}}{{0,2,5,6,7,8}}
\ineq{5}{5}
\end{scope}
\end{tikzpicture}
&&
\begin{tikzpicture}[scale=0.7]
\dwcrossingfillc{2}{olive}
\uwcrossingc{2}{violet}
\drawarray{{0,3,6,7,8,8}}{{2,2,4,7,8}}
\ineq{5}{4}
\draw (6,-.2)--(6,1);
\begin{scope}[shift={(6.5,0)}]
\uwcrossingfillc{1}{olive}
\dwcrossingc{3}{violet}
\drawarray{{2,3,4,10}}{{0,2,5,6,7}}
\ineq{3}{4}
\end{scope}
\end{tikzpicture}\medskip\\
&\overset{\de_{1}}{\longrightarrow}&
\ptra{(0,10]}{[2,7)}{(2,8]}{[0,8)}_{6,2}^{\ge1\dw}=\ptra{(0,10]}{[2,7)}{(2,8]}{[0,8)}_{6,2}\medskip\\
&&
\begin{tikzpicture}[scale=0.7]
\dwcrossingc{2}{olive}
\drawarray{{0,3,4,10}}{{2,2,4,5,6,7}}
\ineq{3}{5}
\draw (5.5,-.2)--(5.5,1);
\begin{scope}[shift={(6,0)}]
\uwcrossingc{1}{olive}
\drawarray{{2,3,6,7,8,8}}{{0,2,7,8}}
\ineq{5}{3}
\end{scope}
\end{tikzpicture}
\end{array}$
\caption{An example of the bijection~\eqref{eq:for_fig1}, where $m=1$ and $n=6$.}
\label{fig:Case1}
\end{figure}

Similarly, if $P\in\P_{A_1\to B_1}$ and $Q\in\P_{A_2\to B_2}$, then $\cro(P,Q)$ must be even,
which proves the first equality in Equation~\eqref{eq:same:des}.
In this case, we construct a sequence of bijections 
 $\de_1\circ\ga_2\circ\dots\circ\de_{2m+1}$:
\begin{align*}\ptra{(x_1,u_1]}{[y_1,v_1)}{(x_2,u_2]}{[y_2,v_2)}_{n}^{\ge 2m+1}
=\ptra{(x_1,u_1]}{[y_1,v_1)}{(x_2,u_2]}{[y_2,v_2)}_{n,0}^{\ge 2m+1\dw}
\overset{\de_{2m+1}}{\longrightarrow}
\ptra{(x_1,u_2]}{[y_1,v_2)}{(x_2,u_1]}{[y_2,v_1)}_{n,1}^{\ge 2m+1\dw}
=\ptra{(x_1,u_2]}{[y_1,v_2)}{(x_2,u_1]}{[y_2,v_1)}_{n,1}^{\ge 2m\uw}& \\
\overset{\ga_{2m}}{\longrightarrow}
\ptra{(x_1,u_1]}{[y_1,v_1)}{(x_2,u_2]}{[y_2,v_2)}_{n,-2}^{\ge 2m\uw}
=\ptra{(x_1,u_1]}{[y_1,v_1)}{(x_2,u_2]}{[y_2,v_2)}_{n,-2}^{\ge 2m-1\dw}& \\
\overset{\de_{2m-1}}{\longrightarrow}\cdots\overset{\de_{1}}{\longrightarrow}
\ptra{(x_1,u_2]}{[y_1,v_2)}{(x_2,u_1]}{[y_2,v_1)}_{n,2m+1}^{\ge 1\dw}
=\ptra{(x_1,u_2]}{[y_1,v_2)}{(x_2,u_1]}{[y_2,v_1)}_{n,2m+1}.&
\end{align*}

Equation~\eqref{eq:Hto2ra} and Lemma~\ref{lem:sum-pairs} now give
$$H^{\ge 2m+1}_{A_1\to B_1,A_2\to B_2}(t,q)=\sum_{n\ge0}t^n\sum_{\cdef\in\ptra{(x_1,u_2]}{[y_1,v_2)}{(x_2,u_1]}{[y_2,v_1)}_{n,2m+1}} q^{\su\bfc+\su\bfd+\su\bfe+\su\bff-nz}=f_{2m+1,A_1,A_2,B_2,B_1}(t,q),
$$
proving Equation~\eqref{eq:same:des}.

\item {\bf Case 3:} endpoints $A_1\prec A_2$ and $B$.
If $P\in\P_{A_1\to B}$ and $Q\in\P_{A_2\to B}$, the parity of $\cro(P,Q)$ is no longer forced by the endpoints, so we consider two cases.
When $r=2m$ for some $m\ge1$, the $r$th crossing is an upward crossing by Lemma~\ref{lem:first_crossing-pairs}, and Lemmas~\ref{lem:even=odd-pairs} and~\ref{lem:ga,de} give a sequence of bijections $\de_1\circ\ga_2\circ\dots\circ\de_{2m-1}\circ\ga_{2m}$:
\begin{align}\nonumber
\ptra{(x_1,u]}{[y_1,v)}{(x_2,u]}{[y_2,v)}_{n}^{\ge 2m}=\ptra{(x_1,u]}{[y_1,v)}{(x_2,u]}{[y_2,v)}_{n,0}^{\ge 2m\uw}
\overset{\ga_{2m}}{\longrightarrow}&\
\ptra{(x_1,u]}{[y_1,v)}{(x_2,u]}{[y_2,v)}_{n,-1}^{\ge 2m\uw}
=\ptra{(x_1,u]}{[y_1,v)}{(x_2,u]}{[y_2,v)}_{n,-1}^{\ge 2m-1\dw}\\ \nonumber
\overset{\de_{2m-1}}{\longrightarrow}&\
\ptra{(x_1,u]}{[y_1,v)}{(x_2,u]}{[y_2,v)}_{n,2}^{\ge 2m-1\dw}
=\ptra{(x_1,u]}{[y_1,v)}{(x_2,u]}{[y_2,v)}_{n,2}^{\ge 2m-2\uw}\\ \label{eq:even:B1=B2}
\overset{\ga_{2m-2}}{\longrightarrow}\cdots\overset{\de_{1}}{\longrightarrow}&\
\ptra{(x_1,u]}{[y_1,v)}{(x_2,u]}{[y_2,v)}_{n,2m}^{\ge 1\dw}
=\ptra{(x_1,u]}{[y_1,v)}{(x_2,u]}{[y_2,v)}_{n,2m}^{\ge 0\uw}
=\ptra{(x_1,u]}{[y_1,v)}{(x_2,u]}{[y_2,v)}_{n,2m},
\end{align}
using again~\eqref{eq:convention_r0-pairs}.
Equation~\eqref{eq:Hto2ra} and Lemma~\ref{lem:sum-pairs} give
$$H^{\ge 2m}_{A_1\to B,A_2\to B}(t,q)=\sum_{n\ge0}t^n\sum_{\cdef\in\ptra{(x_1,u]}{[y_1,v)}{(x_2,u]}{[y_2,v)}_{n,2m}} q^{\su\bfc+\su\bfd+\su\bfe+\su\bff-nz}=f_{2m,A_1,A_2,B,B}(t,q),
$$
proving Equation~\eqref{eq:B1=B2:des} for even $r$.

When $r=2m+1$ for some $m\ge0$, the $r$th crossing is a downward crossing by Lemma~\ref{lem:first_crossing-pairs}, and we get a sequence of bijections $\de_1\circ\ga_2\circ\dots\circ\de_{2m+1}$:
\begin{align}\nonumber
\ptra{(x_1,u]}{[y_1,v)}{(x_2,u]}{[y_2,v)}_{n}^{\ge 2m+1}
=\ptra{(x_1,u]}{[y_1,v)}{(x_2,u]}{[y_2,v)}_{n,0}^{\ge 2m+1\dw}
\overset{\de_{2m+1}}{\longrightarrow}&\
\ptra{(x_1,u]}{[y_1,v)}{(x_2,u]}{[y_2,v)}_{n,1}^{\ge 2m+1\dw}
=\ptra{(x_1,u]}{[y_1,v)}{(x_2,u]}{[y_2,v)}_{n,1}^{\ge 2m\uw}\\ \nonumber
\overset{\ga_{2m}}{\longrightarrow}&\
\ptra{(x_1,u]}{[y_1,v)}{(x_2,u]}{[y_2,v)}_{n,-2}^{\ge 2m\uw}
=\ptra{(x_1,u]}{[y_1,v)}{(x_2,u]}{[y_2,v)}_{n,-2}^{\ge 2m-1\dw}\\  \label{eq:odd:B1=B2}
\overset{\de_{2m-1}}{\longrightarrow}\cdots\overset{\de_{1}}{\longrightarrow}&\
\ptra{(x_1,u]}{[y_1,v)}{(x_2,u]}{[y_2,v)}_{n,2m+1}^{\ge 1\dw}
=\ptra{(x_1,u]}{[y_1,v)}{(x_2,u]}{[y_2,v)}_{n,2m+1},
\end{align}
from where
$$H^{\ge 2m+1}_{A_1\to B_1,A_2\to B_2}(t,q)=\sum_{n\ge0}t^n\sum_{\cdef\in\ptra{(x_1,u}{[y_1,v)}{(x_2,u]}{[y_2,v)}_{n,2m+1}} q^{\su\bfc+\su\bfd+\su\bfe+\su\bff-nz}=f_{2m+1,A_1,A_2,B,B}(t,q),
$$
proving Equation~\eqref{eq:B1=B2:des} for odd $r$.

\item {\bf Case 2:} endpoints $A$ and $B_1\prec B_2$.
We will reduce this case to Case~3 by applying the involution $\nu$, defined above Equation~\eqref{def:nu}, componentwise to each of the two-rowed arrays in a pair. With some abuse of notation, we also denote this map on pairs of two-rowed arrays by $\nu$. It restricts to a bijection
$$\ptra{(x,u_1]}{[y,v_1)}{(x,u_2]}{[y,v_2)}_{n,k}\overset{\nu}{\longleftrightarrow}\ptra{(-v_1,-y]}{[-u_1,-x)}{(-v_2,-y]}{[-u_2,-x)}_{n,-k}$$
for any $k\in\Z$. In the case $k=0$, translating $\nu$ into a map on pairs of paths via the encoding~\eqref{bij:PQto2ra} yields the involution that reflects each path along the line $x+y=0$. In particular, it preserves the number of crossings, so it restricts to a bijection
\begin{equation}\label{eq:nu-pairs}\ptra{(x,u_1]}{[y,v_1)}{(x,u_2]}{[y,v_2)}_{n}^{\ge r}\overset{\nu}{\longleftrightarrow}\ptra{(-v_1,-y]}{[-u_1,-x)}{(-v_2,-y]}{[-u_2,-x)}_{n}^{\ge r}.\end{equation}

The hypothesis $(u_1,v_1)\prec (u_2,v_2)$ implies that the initial points of the reflected paths satisfy $(-v_1,-u_1)\prec (-v_2,-u_2)$, whereas the final point is the same for both paths, namely $(-y,-x)$. This allows us to apply Case~3.

When $r=2m$, Equation~\eqref{eq:even:B1=B2} gives a bijection
$$\de_1\circ\ga_2\circ\dots\circ\de_{2m-1}\circ\ga_{2m}:\ptra{(-v_1,-y]}{[-u_1,-x)}{(-v_2,-y]}{[-u_2,-x)}_{n}^{\ge 2m}\longrightarrow
\ptra{(-v_1,-y]}{[-u_1,-x)}{(-v_2,-y]}{[-u_2,-x)}_{n,2m}.$$
Conjugating by $\nu$ and composing with the map $\swap$ from Equation~\eqref{eq:swap} yields a bijection
$$\swap\circ\nu\circ\de_1\circ\ga_2\circ\dots\circ\de_{2m-1}\circ\ga_{2m}\circ\nu:
\ptra{(x,u_1]}{[y,v_1)}{(x,u_2]}{[y,v_2)}_{n}^{\ge 2m} 
\longrightarrow\ptra{(x,u_2]}{[y,v_2)}{(x,u_1]}{[y,v_1)}_{n,2m}$$
that preserves the sum of the entries. Similarly, when $r=2m+1$, conjugating the bijection~\eqref{eq:odd:B1=B2} with $\nu$ and composing with $\swap$ produces a bijection
$$\swap\circ\nu\circ\de_1\circ\ga_2\circ\dots\circ\de_{2m+1}\circ\nu:\ptra{(x,u_1]}{[y,v_1)}{(x,u_2]}{[y,v_2)}_{n}^{\ge 2m+1}\longrightarrow
\ptra{(x,u_2]}{[y,v_2)}{(x,u_1]}{[y,v_1)}_{n,2m+1}.$$

In both cases, using Equation~\eqref{eq:Hto2ra} and Lemma~\ref{lem:sum-pairs}, we get
$$H^{\ge r}_{A\to B_1,A\to B_2}(t,q)=\sum_{n\ge0}t^n\sum_{\cdef\in\ptra{(x,u_1]}{[y,v_1)}{(x,u_2]}{[y,v_2)}_{n,r}} q^{\su\bfc+\su\bfd+\su\bfe+\su\bff-nz}=f_{r,A,A,B_2,B_1}(t,q),$$
proving Equation~\eqref{eq:A1=A2:des}.

\item {\bf Case 4:} endpoints $A$ and $B$. The map $\swap$ from Equation~\eqref{eq:swap} restricts to a bijection 
$$\ptra{(x,u]}{[y,v)}{(x,u]}{[y,v)}_{n,k}^{\ge r\uw}\overset{\swap}{\longleftrightarrow}\ptra{(x,u]}{[y,v)}{(x,u]}{[y,v)}_{n,-k}^{\ge r\dw}$$ for any $r\ge0$ and $k\in\Z$.
For $r\ge1$, we also have 
$$\ptra{(x,u]}{[y,v)}{(x,u]}{[y,v)}_{n}^{\ge r}=\ptra{(x,u]}{[y,v)}{(x,u]}{[y,v)}_{n}^{\ge r\uw}\sqcup\ptra{(x,u]}{[y,v)}{(x,u]}{[y,v)}_{n}^{\ge r\dw},$$
and so Equation~\eqref{eq:Hto2ra} gives
\begin{align}\label{eq:Hto2raABuw} H^{\ge r}_{A\to B,A\to B}(t,q)&=2\sum_{n\ge0}\,t^n\sum_{\cdef\in\ptra{(x,u]}{[y,v)}{(x,u]}{[y,v)}_{n}^{\ge r\uw}} q^{\su\bfc+\su\bfd+\su\bfe+\su\bff-nz}\\ \label{eq:Hto2raABdw}
&=2\sum_{n\ge0}\,t^n\sum_{\cdef\in\ptra{(x,u]}{[y,v)}{(x,u]}{[y,v)}_{n}^{\ge r\dw}} q^{\su\bfc+\su\bfd+\su\bfe+\su\bff-nz}.
\end{align}

Our next goal is to prove that
\begin{equation}\label{eq:Hrr+1}
H^{\ge r}_{A\to B,A\to B}(t,q)+H^{\ge r+1}_{A\to B,A\to B}(t,q)=2 f_{r+1,A,A,B,B}(t,q)
\end{equation}
for all $r\ge1$.

For arrays with at least $r=2m$ crossings, Lemmas~\ref{lem:even=odd-pairs} and~\ref{lem:ga,de} give bijections
$\de_1\circ\ga_2\circ\dots\circ\de_{2m-1}\circ\ga_{2m}$:
\begin{align}\nonumber
\ptra{(x,u]}{[y,v)}{(x,u]}{[y,v)}_{n}^{\ge 2m\uw}
\overset{\ga_{2m}}{\longrightarrow}&\
\ptra{(x,u]}{[y,v)}{(x,u]}{[y,v)}_{n,-1}^{\ge 2m\uw}
=\ptra{(x,u]}{[y,v)}{(x,u]}{[y,v)}_{n,-1}^{\ge 2m-1\dw}\\ \nonumber
\overset{\de_{2m-1}}{\longrightarrow}&\
\ptra{(x,u]}{[y,v)}{(x,u]}{[y,v)}_{n,2}^{\ge 2m-1\dw}
=\ptra{(x,u]}{[y,v)}{(x,u]}{[y,v)}_{n,2}^{\ge 2m-2\uw}\\ \label{eq:even:A1=A2,B1=B2}
\overset{\ga_{2m-2}}{\longrightarrow}\cdots\overset{\de_{1}}{\longrightarrow}&\
\ptra{(x,u]}{[y,v)}{(x,u]}{[y,v)}_{n,2m}^{\ge 1\dw}
=\ptra{(x,u]}{[y,v)}{(x,u]}{[y,v)}_{n,2m}^{\ge 0\uw}.
\end{align}

Similarly, for arrays with at least $r=2m+1$ crossings, we get bijections $\de_1\circ\ga_2\circ\dots\circ\de_{2m-1}\circ\ga_{2m}\circ\de_{2m+1}$:
\begin{align}\nonumber
\ptra{(x,u]}{[y,v)}{(x,u]}{[y,v)}_{n}^{\ge 2m+1\dw}
\overset{\de_{2m+1}}{\longrightarrow}&\
\ptra{(x,u]}{[y,v)}{(x,u]}{[y,v)}_{n,1}^{\ge 2m+1\dw}
=\ptra{(x,u]}{[y,v)}{(x,u]}{[y,v)}_{n,1}^{\ge 2m\uw}\\ \nonumber
\overset{\ga_{2m}}{\longrightarrow}&\
\ptra{(x,u]}{[y,v)}{(x,u]}{[y,v)}_{n,-2}^{\ge 2m\uw}
=\ptra{(x,u]}{[y,v)}{(x,u]}{[y,v)}_{n,-2}^{\ge 2m-1\dw}\\ \label{eq:odd:A1=A2,B1=B2}
\overset{\de_{2m-1}}{\longrightarrow}\cdots\overset{\de_{1}}{\longrightarrow}&\
\ptra{(x,u]}{[y,v)}{(x,u]}{[y,v)}_{n,2m+1}^{\ge 1\dw}
=\ptra{(x,u]}{[y,v)}{(x,u]}{[y,v)}_{n,2m+1}^{\ge 0\uw}.
\end{align}

In both cases, we can compose these bijections with $\swap\circ\ga_0$:
\begin{equation}\label{eq:swapga0}
\ptra{(x,u]}{[y,v)}{(x,u]}{[y,v)}_{n,r}^{\ge 0\uw}\overset{\ga_{0}}{\longrightarrow}
\ptra{(x,u]}{[y,v)}{(x,u]}{[y,v)}_{n,-r-1}^{\ge 0\uw}
\overset{\swap}{\longrightarrow}
\ptra{(x,u]}{[y,v)}{(x,u]}{[y,v)}_{n,r+1}^{\ge 0\dw}.
\end{equation}

Composing~\eqref{eq:even:A1=A2,B1=B2} with~\eqref{eq:swapga0}, where $r=2m$, and using Equation~\eqref{eq:Hto2raABuw}, we get
$$H^{\ge 2m}_{A\to B,A\to B}(t,q)=2\sum_{n\ge0}t^n\sum_{\cdef\in\ptra{(x,u]}{[y,v)}{(x,u]}{[y,v)}^{\ge0\dw}_{n,2m+1}} q^{\su\bfc+\su\bfd+\su\bfe+\su\bff-nz}
$$
for $m\ge1$.
Similarly, the bijection~\eqref{eq:odd:A1=A2,B1=B2} and Equation~\eqref{eq:Hto2raABdw}, where $r=2m+1$, give
$$H^{\ge 2m+1}_{A\to B,A\to B}(t,q)=2\sum_{n\ge0}t^n\sum_{\cdef\in\ptra{(x,u]}{[y,v)}{(x,u]}{[y,v)}^{\ge0\uw}_{n,2m+1}} q^{\su\bfc+\su\bfd+\su\bfe+\su\bff-nz}.
$$
Adding the last two equations, using the fact that 
$$\ptra{(x,u]}{[y,v)}{(x,u]}{[y,v)}_{n,k}^{\ge 0\uw}\sqcup\ptra{(x,u]}{[y,v)}{(x,u]}{[y,v)}_{n,k}^{\ge 0\dw}=
\ptra{(x,u]}{[y,v)}{(x,u]}{[y,v)}_{n,k}$$
for all $k\neq0$, and applying Lemma~\ref{lem:sum-pairs}, we obtain a proof of Equation~\eqref{eq:Hrr+1} for $r=2m$.

On the other hand, composing~\eqref{eq:odd:A1=A2,B1=B2} (with $m-1$ playing the role of $m$)  with~\eqref{eq:swapga0} (with $r=2m-1$) and using Equation~\eqref{eq:Hto2raABdw}, we get 
$$H^{\ge 2m-1}_{A\to B,A\to B}(t,q)=2\sum_{n\ge0}t^n\sum_{\cdef\in\ptra{(x,u]}{[y,v)}{(x,u]}{[y,v)}^{\ge0\dw}_{n,2m}} q^{\su\bfc+\su\bfd+\su\bfe+\su\bff-nz}
$$
for $m\ge1$. Similarly, the bijection~\eqref{eq:even:A1=A2,B1=B2} and Equation~\eqref{eq:Hto2raABuw}, where $r=2m$, give
$$H^{\ge 2m}_{A\to B,A\to B}(t,q)=2\sum_{n\ge0}t^n\sum_{\cdef\in\ptra{(x,u]}{[y,v)}{(x,u]}{[y,v)}^{\ge0\uw}_{n,2m}} q^{\su\bfc+\su\bfd+\su\bfe+\su\bff-nz}.
$$
Adding the last two equations and applying Lemma~\ref{lem:sum-pairs}, we 
obtain a proof of Equation~\eqref{eq:Hrr+1} for $r=2m-1$.

Solving Equation~\eqref{eq:Hrr+1} for $H^{\ge r}_{A\to B,A\to B}(t,q)$ and iterating, we obtain
$$H^{\ge r}_{A\to B,A\to B}(t,q)=2\left(f_{r+1,A,A,B,B}(t,q)-f_{r+2,A,A,B,B}(t,q)+f_{r+3,A,A,B,B}(t,q)-\cdots\right)$$
which proves Equation~\eqref{eq:A1=A2,B1=B2:des} for $r\ge1$. The case $r=0$ follows immediately from 
Equation~\eqref{eq:Hto2ra} and Lemma~\ref{lem:sum-pairs}.
\end{list1}

\subsection*{Acknowledgments}

The author is grateful to Christian Krattenthaler for posing the question of whether the results from~\cite{part1} could be refined by the number of descents.


\begin{thebibliography}{} 

\bibitem{Andrews} G. E. Andrews, {\it The Theory of Partitions}, Encyclopedia of Mathematics and its
Applications, Vol.\ 2, Addison-Wesley, Reading, Mass., 1976.

\bibitem{CES} S. Corteel, S. Elizalde and C. Savage, Partitions with constrained ranks, in preparation.

\bibitem{part1} S. Elizalde, Counting lattice paths by crossings and major index I: the corner-flipping bijections, preprint, \href{https://arxiv.org/abs/2106.09878}{arXiv:2106.09878}.

\bibitem{Eng} O. Engelberg, On some problems concerning a restricted random walk, {\it J. Appl. Probability 2} (1965), 396--404.

\bibitem{Feller} W. Feller, The numbers of zeros and of changes of sign in a symmetric random walk, {\it Enseign. Math.} (2) 3 (1957), 229--235. 

\bibitem{Feller-book} W. Feller, {\it An introduction to probability theory and its applications, Vol. I}, third edition, John Wiley \& Sons, Inc., New York-London-Sydney 1968. 

\bibitem{Fisher} M. E. Fisher, Walks, walls, wetting, and melting, {\it J. Statist. Phys.} 34 (1984), 667--729.


\bibitem{FH} J. F\"urlinger and J. Hofbauer, $q$-Catalan numbers, {\it J. Combin. Theory Ser. A} 40 (1985), 248--264.


\bibitem{GV} I. Gessel and G. Viennot, Binomial determinants, paths, and hook length formulae, {\it Adv. in Math.} 58 (1985), 300--321. 

\bibitem{KW} M. Kern and S. Walter, 
Ballot theorem and lattice path crossings, {\it Canad. J. Statist.} 6 (1978), 87--90. 

\bibitem{Kul} D. M. Kulkarni, Counting of paths and coefficients of Hilbert polynomial of a determinantal ideal, {\it Discrete Math.} 154 (1996), 141--151.

\bibitem{Krat-turns0} C. Krattenthaler, Counting nonintersecting lattice paths with turns, {\it S\'em. Lothar. Combin.} 34 (1995), Art.\ B34i, 17 pp.

\bibitem{Krat} C. Krattenthaler, Lattice path enumeration, {\it Handbook of enumerative combinatorics}, 589--678, Discrete Math. Appl., CRC Press, Boca Raton, FL, 2015. 

\bibitem{Krat-turns} C. Krattenthaler, The enumeration of lattice paths with respect to their number of turns, {\it Advances in combinatorial methods and applications to probability and statistics}, 29--58, Stat. Ind. Technol., Birkh\"auser Boston, Boston, MA, 1997. 

\bibitem{Krat-nonint} C. Krattenthaler, The major counting of nonintersecting lattice paths and generating functions for tableaux, {\it Mem. Amer. Math. Soc.} 115 (1995), no.\ 552, 109 pp.

\bibitem{KM} C. Krattenthaler and S. G. Mohanty, 
On lattice path counting by major index and descents,
{\it European J. Combin.} 14 (1993), 43--51. 

\bibitem{Lin} B. Lindstr\"om, On the vector representations of induced matroids. {\it Bull. London Math. Soc.} 5 (1973), 85--90.

\bibitem{Mac} P.A. MacMahon, {\it Combinatory Analysis}, Cambridge Univ. Press, London,
1915--1916. Reprinted, Chelsea, New York, 1960.

\bibitem{Moh} S. G. Mohanty, {\it Lattice path counting and applications}, Probability and Mathematical Statistics, Academic Press, New York-London-Toronto, 1979.

\bibitem{SaSa}  B. E. Sagan and C. D. Savage, Mahonian pairs, {\it J. Combin. Theory Ser. A} 119 (2012), 526--545.

\bibitem{Sen} K. Sen, On some combinatorial relations concerning the symmetric random walk, {\it 
Magyar Tud. Akad. Mat. Kutat\'o Int. K\"ozl.} 9 (1965), 335--357. 

\bibitem{SeoYee} S. Seo and A. J. Yee, Enumeration of partitions with prescribed successive rank parity blocks, {\it J. Combin. Theory Ser. A} 158 (2018), 12--35. 

\bibitem{Spivey} M. Z. Spivey, Enumerating lattice paths touching or crossing the diagonal at a given number of lattice points, {\it Electron. J. Combin.} 19(3) (2012), \#P24, 6 pp.

\end{thebibliography}
\end{document}